\documentclass[11pt,a4paper]{amsart}
\usepackage{amsmath,amsthm,amssymb,latexsym}
\usepackage{graphicx}

\newcommand{\bb}{{\mathbb{B}}}
\newcommand{\bd}{{\mathbb{D}}}
\newcommand{\bn}{{\mathbb{N}}}
\newcommand{\br}{{\mathbb{R}}}
\newcommand{\bp}{{\mathbb{P}}}
\newcommand{\bz}{{\mathbb{Z}}}
\newcommand{\bc}{{\mathbb{C}}}
\newcommand{\bo}{{\mathbb{O}}}

\newcommand{\cf}{{\mathcal{F}}}
\newcommand{\ch}{{\mathcal{H}}}
\newcommand{\cl}{{\mathcal{L}}}
\newcommand{\ce}{{\mathcal{E}}}
\newcommand{\cv}{{\mathcal{V}}}

\renewcommand{\a}{\alpha}
\renewcommand{\b}{\beta}
\renewcommand{\l}{\lambda}
\newcommand{\s}{\sigma}

\renewcommand{\d}{\delta}

\renewcommand{\o}{\omega}

\newcommand{\g}{\gamma}
\renewcommand{\gg}{\Gamma}
\newcommand{\eps}{\varepsilon}
\newcommand{\z}{\zeta}

\newcommand{\nt}{\noindent}

\newcommand{\ag}{A(\gg)\,}

\newcommand{\whe}{\widehat e}
\newcommand{\wte}{\widetilde e}

\DeclareMathOperator{\lc}{span}
\DeclareMathOperator{\ran}{rank}
\DeclareMathOperator{\diag}{diag}

\allowdisplaybreaks
\numberwithin{equation}{section}

\newtheorem{theorem}{Theorem}[section]

\newtheorem{corollary}[theorem]{Corollary}
\newtheorem{proposition}[theorem]{Proposition}

\theoremstyle{definition}
\newtheorem{definition}[theorem]{Definition}
\newtheorem{remark}[theorem]{Remark}

\newtheorem{example}[theorem]{Example}

\begin{document}

\title[Infinite graphs]
{Spectra of infinite graphs with tails}
\author[L. Golinskii]{L. Golinskii}

\address{Mathematics Division, Institute for Low Temperature Physics and
Engineering, 47 Lenin ave., Kharkov 61103, Ukraine}
\email{golinskii@ilt.kharkov.ua}

\date{\today}


\maketitle

\begin{abstract}
We compute explicitly (modulo solutions of certain algebraic equations) the spectra of
infinite graphs obtained by attaching one or several infinite paths to some vertices of
certain finite graphs. The main result concerns a canonical form of the adjacency matrix
of such infinite graphs. A complete answer is given in the case when the number of
attached paths to each vertex is the same.
\end{abstract}

\section*{Introduction and preliminaries}
\label{s0}

\subsection{Graph theory}

We begin with rudiments of the graph theory. For the sake of simplicity we restrict ourselves
with simple, connected, undirected, finite or infinite (countable) weighted graphs, although
the main result holds for weighted multigraphs and graphs with loops as well.
We will label the vertex set $\cv(\gg)$ by positive integers $\bn=\{1,2,\ldots\}$,
$\{v\}_{v\in \cv}=\{j\}_{j=1}^\o$, $\o\le\infty$. The symbol $i\sim j$ means that the vertices
$i$ and $j$ are incident, i.e., $\{i,j\}$ belongs to the edge set $\ce(\gg)$.
A graph $\gg$ is weighted if a positive number $d_{ij}$ (weight) is assigned to each edge
$\{i,j\}\in\ce(\gg)$. In case $d_{ij}=1$ for all $i,j$, the graph is unweighted.

The degree (valency) of a vertex $v\in\cv(\gg)$ is a number $\g(v)$ of edges emanating from $v$.
A graph $\gg$ is said to be locally finite, if $\g(v)<\infty$ for all $v\in\cv(\gg)$, and
uniformly locally finite, if $\sup_{\cv}\g(v)<\infty$.

The spectral graph theory deals with the study of spectra and spectral properties of certain
matrices related to graphs (more precisely, operators generated by such matrices in the standard
basis $\{e_k\}_{k\in\bn}$ and acting in the corresponding Hilbert spaces $\bc^n$ or
$\ell^2=\ell^2(\bn)$). One of the most notable of them is the {\it adjacency matrix} $A(\gg)$
\begin{equation}\label{adjmat}
A(\gg)=\|a_{ij}\|_{ij=1}^\o, \quad
a_{ij}=\left\{
  \begin{array}{ll}
    d_{ij}, & \{i,j\}\in\ce(\gg); \\
    0, & \hbox{otherwise.}
  \end{array}
\right.
\end{equation}
The corresponding adjacency operator will be denoted by the same symbol. It acts as
\begin{equation}\label{adjop}
A(\gg)\,e_k=\sum_{j\sim k} a_{jk}\,e_j, \qquad k\in\bn.
\end{equation}
Clearly, $\ag$ is a symmetric, densely-defined linear operator, whose domain is the set of all finite
linear combinations of the basis vectors. The operator $A(\gg)$ is bounded and self-adjoint in $\ell^2$,
as long as the graph $\gg$ is uniformly locally finite.

Whereas the spectral theory of finite graphs is very well established (see, e.g.,
\cite{Bap, BrHae, Chung97, CDS80}), the corresponding theory for infinite graphs is in its infancy. 
We refer to \cite{M82, MoWo89, SiSz} for the basics of this theory. In contrast to the general
consideration in \cite{MoWo89}, our goal is to carry out a complete spectral analysis (canonical models for the adjacency operators and computation of the spectrum) for a class of infinite graphs which loosely speaking can be called ``finite graphs with tails attached to them''. To make the notion precise, we define first an operation
of coupling well known for finite graphs (see, e.g., \cite[Theorem 2.12]{CDS80}).

\begin{definition}\label{coupl}
Let $\gg_k$, $k=1,2$, be two weighted graphs with no common vertices, with the vertex sets and edge
sets $\cv(\gg_k)$ and $\ce(\gg_k)$, respectively, and let $v_k\in \cv(\gg_k)$. A weighted graph
$\gg=\gg_1+\gg_2$ will be called a {\it coupling by means of the bridge $\{v_1,v_2\}$ of weight} $d$ if
\begin{equation}\label{defcoup}
\cv(\gg)=\cv(\gg_1)\cup \cv(\gg_2), \qquad \ce(\gg)=\ce(\gg_1)\cup \ce(\gg_2)\cup \{v_1,v_2\}.
\end{equation}
So we join $\gg_2$ to $\gg_1$ by the new edge of weight $d$ between $v_2$ and $v_1$.
\end{definition}

If the graph $\gg_1$ is finite, $V(\gg_1)=\{1,2,\ldots,n\}$, and $V(\gg_2)=\{j\}_{j=1}^\o$,
we can with no loss of generality put $v_1=n$, $v_2=1$, so the adjacency matrix $A(\gg)$ can
be written as a block matrix
\begin{equation}\label{adjcoup}
A(\gg)=\begin{bmatrix}
A(\gg_1) & I_d \\
I_d^*& A(\gg_2) &
\end{bmatrix}, \qquad
I_d=\begin{bmatrix}
0 & 0 & 0 & \ldots \\
\vdots & \vdots & \vdots &  \\
0 & 0 & 0 & \ldots \\
d & 0 & 0 & \ldots
\end{bmatrix}.
\end{equation}
If $\gg_2=\bp_\infty(\{a_j\})$, the one-sided weighted infinite path, we can view the coupling
$\gg=\gg_1+\bp_\infty(\{a_j\})$ as a finite graph with the tail. This is exactly the class of graphs
we will be dealing with in the paper. Each such graph has a finite number of essential ramification
nodes (a vertex $v$ is an essential ramification node if $\g(v)\ge3$, see \cite{Be09}).

\begin{picture}(300, 100)

\put(50, 50){\circle{80}} \put(44, 44){\Large{$\Gamma_1$}}

\multiput(72, 50) (50,0) {3} {\circle* {4}}
\multiput(74, 50) (50,0) {2} {\line(1,0) {46}}
\put(72, 56) {$N$} \put(104, 56) {$N+1$} \put(154, 56) {$N+2$}
\multiput(190, 50) (10, 0) {3} {\circle*{2}}

\end{picture}

The spectral theory of infinite graphs with one or several rays attached to certain finite graphs 
was initiated in \cite{Le-star, LeNi-dan, LeNi-umzh, Niz14} wherein a number of particular examples 
of unweighted (background) graphs is examined. We argue in the spirit of \cite{Br07, Br071, Si96} and 
suggest a quite general canonical form for the adjacency matrix of such graphs and supplement to the 
list of the examples. As a matter of fact, the algorithm applies not only to adjacency matrices, but 
to both Laplacians on graphs of such type. We also find a canonical form of the adjacency operator for 
graphs obtained from an arbitrary finite graph by attaching an equal number of infinite rays to 
{\it each} of its vertices.

\subsection{Jacobi matrices}

Under {\it Jacobi matrices} we mean here one-sided infinite matrices of the form
\begin{equation}\label{defjac}
J=J(\{b_j\}, \{a_j\})=
\begin{bmatrix}
 b_1 & a_1 & & \\
 a_1 & b_2 & a_2 & \\
     & a_2 & b_3 & \ddots & \\
     &     & \ddots & \ddots
\end{bmatrix}, \quad b_j\in\br, \quad a_j>0,
\end{equation}
which generate linear operators (called the Jacobi operators) on the Hilbert space $\ell^2(\bn)$.
The matrix
\begin{equation}\label{free}
J_0:=
\begin{bmatrix}
 0 & 1 & 0 & 0 & \\
 1 & 0 & 1 & 0 & \\
 0 & 1 & 0 & 1 & \\
     & \ddots  & \ddots & \ddots & \ddots
\end{bmatrix}
\end{equation}
called a {\it discrete Laplacian} or a {\it free Jacobi matrix}, is of particular interest in 
the sequel.

Given two Jacobi matrices $J_k=J(\{\b_j^{(k)}\}, \{\a_j^{(k)}\})$, $k=1,2$, the matrix $J_2$
is called a {\it truncation} of $J_1$ (and $J_1$ is an {\it extension} of $J_2$) if
\begin{equation*}
\b_j^{(2)}=\b_{j+q}^{(1)}, \quad \a_j^{(2)}=\a_{j+q}^{(1)}, \qquad j\in\bn,
\end{equation*}
for some $q\in\bn$. In other words, $J_2$ is obtained from $J_1$ by deleting the first $q$
rows and columns. If $J_2=J_0$, $J_1$ is said to be a {\it Jacobi matrix of finite rank} or an
{\it eventually free Jacobi matrix}.

For the class of Jacobi matrices of finite rank the complete spectral analysis is available at
the moment (see \cite{DaSi06, Ko14}). A basic in perturbation theory object known as the
{\it perturbation determinant} \cite{GK69} plays a key role. Given bounded linear operators $T_0$ 
and $T$ on the Hilbert space such that $T-T_0$ is a nuclear
operator, the perturbation determinant is defined by
\begin{equation}\label{perdet}
L(\l;T,T_0):=\det(I+(T-T_0)R(\l,T_0)), \qquad R(\l,T_0):=(T_0-\l)^{-1}
\end{equation}
is the resolvent of operator $T_0$, an analytic operator-function on the resolvent set $\rho(T_0)$.

The perturbation determinant is designed for the spectral analysis of the perturbed operator $T$, once
the spectral analysis for $T_0$ is available. In particular, the essential spectra of $T$ and $T_0$
agree, and the discrete spectrum of $T$ is exactly the zero set of the analytic function $L$ on $\rho(T_0)$,
at least if the latter is a domain, i.e., a connected, open set in the complex plane.

In the simplest case $\ran(T-T_0)<\infty$ the perturbation determinant is the standard finite
dimensional determinant. Indeed, now
$$ (T-T_0)h=\sum_{k=1}^p \langle h, \varphi_k\rangle\,\psi_k, \quad
(T-T_0)R(\l,T_0)h=\sum_{k=1}^p \langle h, R^*(\l,T_0)\varphi_k\rangle\,\psi_k,
$$
so $L$ can be computed by the formula
\begin{equation}\label{compd}
L(\l; T,T_0)=\det\|\delta_{ij}+\langle R(\l,T_0)\psi_i,\varphi_j\rangle\|_{i,j=1}^p.
\end{equation}

Our particular concern is $T_0=J_0$, the free Jacobi matrix. Its resolvent matrix in the
standard basis in $\ell^2$ is given by
\begin{equation}\label{resfree}
R\Bigl(z+\frac1{z}\,, \,J_0\Bigr)=\|r_{ij}(z)\|_{i,j=1}^\infty, \quad
r_{ij}(z)=\frac{z^{|i-j|}-z^{i+j}}{z-z^{-1}}\,,  \quad z\in\bd,
\end{equation}
see, e.g., \cite{KiSi03}. If $T=J$ is a Jacobi matrix of finite rank $p$, we end up with the computation
of the ordinary determinant \eqref{compd} of order $p$.

It is instructive for the further usage to compute two simplest perturbation determinants for
$\ran\,(J-J_0)=1$ and $2$.

\begin{example}\label{pd1}
Let
$$ J=J(\{b_j\}, \{1\}): \quad b_j=0, \quad j\not=q,
$$
so $J-J_0=\langle\cdot,e_q\rangle\,b_q e_q$. Hence by \eqref{resfree} and \eqref{compd}
\begin{equation}\label{rank1}
\widehat L(z) :=L\Bigl(z+\frac1{z}\,; J,J_0\Bigr)=1+b_q r_{qq}(z)
=1-b_q z\,\frac{z^{2q}-1}{z^2-1}\,.
\end{equation}

Similarly, let
$$ J=J(\{0\}, \{a_j\}): \quad a_j=1, \quad j\not=q,
$$
so $J-J_0=\langle\cdot,e_q\rangle\,(a_q-1)\,e_{q+1}+\langle\cdot,e_{q+1}\rangle\,(a_q-1)\,e_{q}$.
Hence by \eqref{resfree} and \eqref{compd}
\begin{equation}\label{rank2}
\begin{split}
\widehat L(z) &=\begin{vmatrix}
1+(a_q-1)\, r_{q,q+1}(z) & (a_q-1)\, r_{qq}(z) \\
(a_q-1)\, r_{q+1,q+1}(z) & 1+(a_q-1)\, r_{q+1,q}(z)
\end{vmatrix}
\\
&=1+(1-a_q^2) z^2\,\frac{z^{2q}-1}{z^2-1}\,.
\end{split}
\end{equation}
\end{example}

\medskip

In the Jacobi matrices setting there is yet another way of computing perturbation determinants
based on the so-called Jost solution and Jost function (see, e.g., \cite[Section 3.7]{SiSz}).
Consider the basic recurrence relation for the Jacobi matrix $J$
\begin{equation}\label{3term}
a_{n-1}y_{n-1}+b_ny_n+a_ny_{n+1}=\Bigl(z+\frac1{z}\Bigr)\,y_n,  \quad z\in\bd, \quad n\in\bn,
\end{equation}
where we put $a_0=1$. Its solution $y_n=u_n(z)$ is called the {\it Jost solution} if
\begin{equation}\label{js1}
\lim_{n\to\infty} z^{-n}u_n(z)=1, \qquad z\in\bd.
\end{equation}
The function $u=u_0$ in this case is called the {\it Jost function}.

The Jost solution definitely exists for finite rank Jacobi matrices. Indeed, let
$$ a_{q+1}=a_{q+2}=\ldots=1, \qquad b_{q+1}=b_{q+2}=\ldots=0. $$
One can put $u_k(z)=z^k$, $k=q+1,q+2,\ldots$ and then determine $u_q, u_{q-1},\ldots,u_0$
consecutively from \eqref{3term}. So,
\begin{equation}\label{js2}
\begin{split}
a_q\,u_q(z) &=z^q, \\
a_{q-1}a_q\,u_{q-1}(z) &=\a_q\,z^{q+1}-b_qz^q+z^{q-1}, \quad \a_q:=1-a_q^2
\end{split}
\end{equation}
etc., and in general
$$ u_{q-k}(z)=\sum_{j=-k}^k \b_{q,j}z^{q+j}, \qquad k=0,1,\ldots,q, \quad \b_{q,j}\in\br, \ \ \b_{q,-k}=1. $$
In particular, for $q=1$
\begin{equation}\label{jf1}
a_1\,u(z)=\a_1\,z^2-b_1z+1,
\end{equation}
and for $q=2$
\begin{equation}\label{jf2}
a_1a_2\,u(z) =\a_2\,z^4-(b_2+b_1\a_2)\,z^3 +(\a_1+\a_2+b_1b_2)\,z^2 -(b_1+b_2)\,z+1.
\end{equation}

The relation between the perturbation determinant and the Jost function is given by
\begin{equation}\label{pdfj}
u(z)=\prod_{j=1}^\infty a_j^{-1}\cdot \widehat L(z),
\end{equation}
see, e.g., \cite{KiSi03}, and such recursive way of computing perturbation determinants is sometimes far
easier than computing ordinary determinants \eqref{compd}, especially for large enough ranks of perturbation.

\begin{example}\label{jfcomp}
Let $J=J(\{b_j\}, \{a_j\})$ be a Jacobi matrix such that
$$ b_j=0, \quad j\not=1, \qquad a_j=1, \quad j\not=q. $$
We have $u_{q+j}(z)=z^{q+j}$, $j=1,2\ldots$,
\begin{equation*}
\begin{split}
a_q u_q(z) &=z^q, \quad a_qu_{q-1}(z) =\a_q\,z^{q+1}+z^{q-1}, \\
a_qu_{q-2} &=\a_q\,(z^{q+2}+z^q)+z^{q-2},
\end{split}
\end{equation*}
and, by induction,
\begin{equation}\label{jf21}
a_qu_{q-k}(z)=\a_q\,z^{q-k+2}\,\frac{z^{2k}-1}{z^2-1}\,,\qquad k=1,2,\ldots,q-1.
\end{equation}

Next, for $q=1$ we have exactly \eqref{jf1}, so let $q\ge2$.
The recurrence relation \eqref{3term} with $n=1$ gives
$$ a_q u(z)+b_1a_q u_1(z)+a_q u_{2}(z)=\Bigl(z+\frac1{z}\Bigr)\,a_qu_1(z), $$
and so we come to the following expression for the Jost function
\begin{equation}\label{jf3}
a_qu(z)(z^2-1)=\a_q\,(z-b_1)\,z^{2q+1}-b_1a_q^2 z^3+a_q^2 z^2+b_1 z-1.
\end{equation}

Similarly, for the Jacobi matrix $J=J(\{b_j\}, \{a_j\})$ with
$$ b_j=0, \ \  j\not=q; \quad a_j=1, \quad j\not=1 $$
one has
\begin{equation}\label{jf4}
a_1u(z)=-b_q\,\frac{z^{2q+1}+\a_1z^{2q-1}-\a_1z^3-z}{z^2-1}+\a_1z^2+1.
\end{equation}
For the Jacobi matrix $J=J(\{0\}, \{a_j\})$ with $a_j=1$, $j\not=1$, $q$, the Jost function 
is given by
\begin{equation}\label{jf5}
a_1a_q u(z)=\a_q\,\frac{z^{2q+2}+\a_1z^{2q}-\a_1z^4-z^2}{z^2-1}+ \a_1z^2+1.
\end{equation}
\end{example}

\medskip

The spectral theorem for finite rank Jacobi matrices provides a complete description of the spectral
measure of $J$ \cite{DaSi06}.

{\bf Theorem DS}. Let $J=J(\{b_j\}, \{a_j\})$ be a Jacobi matrix of finite rank
$$ a_{q+1}=a_{q+2}=\ldots=1, \qquad  b_{q+1}=b_{q+2}=\ldots=0, $$
and $u=u_0(J)$ be its Jost function. Then
\begin{itemize}
  \item $u$ is a real polynomial of degree $\deg u\le 2q$ (the Jost polynomial), 
  $\deg u=2q$ if and only if $a_q\not=1$.
  \item All roots of $u$ in the unit disk $\bd$ are real and simple, $u(0)\not=0$. A number $\l_j$
  is an eigenvalue of $J$ if and only if
\begin{equation}\label{zhuk}
\l_j=z_j+\frac1{z_j}, \qquad z_j\in(-1,1), \quad u(z_j)=0.
\end{equation}
  \item The spectral measure $\s(J)$ is of the form
\begin{equation}
\s(J,dx)=\s_{ac}(J,dx)+\s_d(J,dx)=w(x)\,dx+\sum_{j=1}^N \s_j\d(\l_j),
\end{equation}
where
$$ w(x):=\frac{\sqrt{4-x^2}}{2\pi|u(e^{it})|^2}\,, \quad x=2\cos t, \quad \s_j=\frac{z_j(1-z_j^{-2})^2}{u'(z_j)u(1/z_j)}\,. $$
\end{itemize}
Note that $|u(e^{it})|^2=Q(x)$, $x=2\cos t$, $Q$ is a real polynomial of the same degree as the
Jost polynomial $u$.

\medskip

The Jacobi matrices arise in the spectral graph theory thanks to the relation for the adjacency matrix $A(\bp_\infty(\{a_j\}))$ of the weighted path
\begin{equation}\label{adjpathw}
A(\bp_\infty(\{a_j\})):=J(\{0\}, \{a_j\}).
\end{equation}
In case of the unweighted path we have
\begin{equation}\label{adjpath}
A(\bp_\infty)=J_0.
\end{equation}
Note that the adjacency matrix for the finite (unweighted) path $\bp_m$ with $m$ vertices is
the finite Jacobi matrix $J(\{0\}, \{1\})$ of order $m$. The spectrum of this matrix is well
known \cite[p. 9]{BrHae}
\begin{equation}\label{specpath}
\s(\bp_m)=\Bigl\{2\cos\frac{\pi j}{m+1}\Bigr\}_{j=1}^m.
\end{equation}

It follows from \eqref{adjcoup} that for an arbitrary finite weighted graph $G$
\begin{equation}
A(G+\bp_\infty(\{a_j\}))=
\begin{bmatrix}
A(G) & I_d \\
I_d^*& J(\{0\},\{a_j\}) &
\end{bmatrix}.
\end{equation}

\medskip

Sometimes two-sided Jacobi matrices
\begin{equation}\label{defjac2}
J=J(\{b_j\}_{j\in\bz}, \{a_j\}_{j\in\bz})=
\begin{bmatrix}
\ddots & \ddots & \ddots &  & \\
 & a_{-1} & b_0 & a_0 &  & \\
 & & a_0 & b_1 & a_1 & & \\
 & &  & a_1 & b_2 & a_2 & \\
 & &  & & \ddots & \ddots & \ddots
\end{bmatrix},
\end{equation}
$b_j\in\br$, $a_j>0$, which generate linear operators on the two-sided $\ell^2(\bz)$,
arise in the canonical models for certain infinite graphs. We say
that such matrix has a finite rank if there are integers $N^{\pm}$, $N^-<N^+$ such that
$$ b_n=0, \quad n\notin[N^-, N^+], \quad a_n=1, \quad n\notin[N^-, N^+-1]. $$
In this case one has a pair of Jost solutions $\{u_n^{\pm}\}$ of the basic recurrence relation
\eqref{3term} (with $n\in\bz$) so that
\begin{equation}\label{jostsol2}
u_n^+(z)=z^n, \quad n\ge N^+, \quad u_n^-(z)=z^{-n}, \quad n\le N^-.
\end{equation}
The Wronskian of two solutions $\{f_n(z)\}_{n\in\bz}$ and $\{g_n(z)\}_{n\in\bz}$ of \eqref{3term}
is defined as
\begin{equation}\label{wron}
[f,g]:=a_n\bigl(f_n(z)g_{n+1}(z)-f_{n+1}(z)g_n(z)\bigr)
\end{equation}
(the right hand side does not actually depend on $n$). We denote by
\begin{equation}\label{wronjost}
w(z):=[u^+(z),u^-(z)], \qquad z\in\bd
\end{equation}
the Wronskian of two Jost solutions of \eqref{3term}.

It is well known that the spectrum of a two-sided Jacobi matrix $J$ of the finite rank is
$$ \s(J)=[-2,2]\cup\s_d(J), \qquad \s_d(J)=\{\l_j\}_{j=1}^N, $$
and for the eigenvalues $\l_j$ the following relation (see, e.g., \cite[Theorem~10.4]{Tjo})
\begin{equation}\label{eigen2}
\exists \,h\in\ell^2(\bz): \ Jh=\Bigl(\z+\frac1{\z}\Bigr)h \,\Longleftrightarrow\,
w(\z)=0, \quad \z\in\bd,
\end{equation}
holds. So the eigenvalues of $J$ are exactly the Zhukovsky images of zeros of wronskian $w$ \eqref{wronjost} in
the unit disk.

\begin{example}\label{twojm1}
Let
\begin{equation}
J=J(\{0\}_{j\in\bz}, \{a_j\}_{j\in\bz}), \quad a_j=1, \ \ j\not=0.
\end{equation}
We now have $N^-=0$, $N^+=1$. A simple computation gives
$$ w(z)=\frac1{a_0z}-a_0z, $$
and so
$$ \s_d(J)=\pm\Bigl(a_0+\frac1{a_0}\Bigr), \qquad a_0>1, $$
and $\s_d(J)$ is empty for $0<a_0\le1$.
\end{example}

\begin{example}\label{twojm2}
Let
\begin{equation}
J=J(\{0\}_{j\in\bz}, \{a_j\}_{j\in\bz}), \quad a_j=1, \ \ j\not=\pm1.
\end{equation}
We now have $N^-=-1$, $N^+=2$. A simple computation gives
\begin{equation*}
\begin{split}
w(z)&=\a_1\a_{-1}z^3+\Bigl(\frac{\a_1}{a_{-1}}+\frac{\a_{-1}}{a_{1}}-\frac1{a_1a_{-1}} \Bigr)\,z
+\frac{1}{a_1a_{-1}z}, \\
\a_k &:=\frac1{a_k}-a_k, \quad k=\pm1,
\end{split}
\end{equation*}
and so to determine the discrete spectrum $\s_d(J)$ one has to analyze the roots of the
quadratic equation
\begin{equation}\label{quadr}
Q(y)=Ay^2-By+1=0, \quad A:=(a_{-1}^2-1)(a_1^2-1), \quad B=a_1^2+a_{-1}^2-1.
\end{equation}
Precisely, each root $y_0$ of this equation in $(-1,1)$ generates two symmetric eigenvalues
$$ \l_{\pm}:=\pm\Bigl(\sqrt{y_0}+\frac1{\sqrt{y_0}}\Bigr). $$
\end{example}

\subsection{Canonical form for adjacency matrices}

It is well known that each bounded and self-adjoint linear operator on a separable Hilbert
space is unitarily equivalent to an infinite orthogonal sum of Jacobi operators. We suggest
here a ``canonical'' form for certain infinite adjacency matrices and the algorithm of their
reducing to such form.

\begin{theorem}\label{algorithm}
Let $A$ be a block matrix in $\ell^2$,
\begin{equation}\label{adjmat}
A=\begin{bmatrix}
\mathcal{A} & I_d \\
I_d^*& J &
\end{bmatrix}, \qquad
I_d=\begin{bmatrix}
0 & 0 & 0 & \ldots \\
\vdots & \vdots & \vdots &  \\
0 & 0 & 0 & \ldots \\
d & 0 & 0 & \ldots
\end{bmatrix},
\end{equation}
where $\mathcal{A}=\|a_{ij}\|_{i,j=1}^n$ is a real symmetric matrix of order $N$,
$J=J(\{\beta_j\}, \{\alpha_j\})$ is a Jacobi matrix. Then $A$ can be reduced to the
block diagonal form
\begin{equation}\label{canform}
A\simeq \begin{bmatrix}
\widehat{\mathcal A} &  \\
  & \widehat J &
\end{bmatrix}
\end{equation}
where $\widehat{\mathcal{A}}$ is a real symmetric matrix of order $n-k$, for some
$1\le k\le n$, and the Jacobi matrix $\widehat J$ is the extension of $J$.
In other words, there is a unitary operator $U$ in $\ell^2$ such that
\begin{equation}
U^{-1}AU=\widehat{\mathcal{A}}\oplus \widehat J.
\end{equation}
\end{theorem}

\begin{corollary}
Given a finite weighted graph $G$, the adjacency operator of the coupling $\gg=G+\bp_\infty(\{a_j\})$
is unitarily equivalent to the orthogonal sum
\begin{equation}
U^{-1}A(\gg)U= F(\gg)\oplus J(\gg)
\end{equation}
of a finite dimensional operator $F(\gg)$ and a Jacobi operator $J(\gg)$ which is the extension of
$J(\{0\},\{a_j\})$. If $\bp$ is unweighted, $J(\gg)$ is of finite rank.
\end{corollary}

We call $F(\gg)$ a {\it finite-dimensional component} of the coupling $\gg$ and $J(\gg)$ its
{\it Jacobi component}.

In particular, the spectrum of $\gg$ is
\begin{equation}\label{spect}
\s(\gg)=[-2,2]\cup\s_d(\gg), \qquad \s_d(\gg)=\{\l_j\}_{j=1}^\o, \ \ \o<\infty,
\end{equation}
is a union of eigenvalues of $F(\gg)$ and $J(\gg)$.  By Theorem DS, the spectral analysis of the
graphs with tails amounts thereby to finding their finite-dimensional and Jacobi components,
computing the Jost polynomial for $J(\gg)$ and solving two algebraic equations, the characteristic
equation for $F(\gg)$ and the Jost equation for $J(\gg)$.

\begin{remark}\label{multicoupl}
Given a finite graph $G$, one can attach $p\ge1$ copies of the infinite path $\bp_\infty$ to some vertex
$v\in\cv(G)$. Although the graph $\gg$ thus obtained is not exactly the coupling in the sense of
Definition \ref{coupl}, its adjacency operator acts similarly to one for the coupling. Indeed, it is not
hard to see that
\begin{equation}
A(\gg)=\begin{bmatrix}
A(G) & I_{\sqrt{p}} \\
I_{\sqrt{p}}^*& J_0 &
\end{bmatrix} \bigoplus\Bigl(\bigoplus_{i=1}^{p-1} J_0\Bigr).
\end{equation}
Hence Theorem \ref{algorithm} applies, and the spectral analysis of such graph can be accomplished.
\end{remark}

\medskip

The algebraic equations which encounter later on cannot in general be solved explicitly. We can only
determine how many roots (if any) they have in $(-1,1)$, by means of the following well-known result
(see, e.g., \cite[p.~41]{PoSz}).

{\bf Theorem (Descarte's rule)}. Let $a(x)=a_0x^n+\ldots+a_n$ be a real polynomial. Denote by $\mu(a)$
the number of its positive roots, and $\nu(a)$ the number of the sign changes in the sequence
$\{a_0,\ldots,a_n\}$ of its coefficients (the zero coefficients are not taken into account).
Then $\nu(a)-\mu(a)$ is a nonnegative even number.

{\bf Acknowledgement}. I thank Professor Yu. Samoilenko for drawing my attention to the spectral graph theory
and putting forward some particular problems related to infinite graphs with tails and their spectra.

\section{Proof of the main result}
\label{s1}

Before starting the proof we note that the algorithm of constructing a new basis suggested
below is even more important than the result itself. To determine the spectra of graphs in the
next Section we will have to carry it out by hand. The result of Theorem \ref{algorithm} only
guarantees that the procedure can be accomplished with enough paper and patience.

The algorithm applies in a number of situations beyond graphs with one tail (for instance, for
the chain of cycles, the ladder with missing rungs etc.). We will elaborate on this topic in our
forthcoming papers.

\medskip
\nt {\it Proof of Theorem \ref{algorithm}.}

We construct an orthogonal basis $\{\wte_k\}_{k\ge1}$ in $\ell^2$ so that the matrix of $A$ in this
basis has the canonical form \eqref{canform}. We take $\wte_k=e_k$, $k=n,n+1,\ldots$, so
\begin{equation}\label{recurr}
\begin{split}
A\wte_{n+1} &=d\wte_n+\b_1\wte_{n+1}+\a_1\wte_{n+2}, \\
A\wte_{n+k} &=\a_{k-1}\wte_{n+k-1}+\b_k\wte_{n+k}+\a_k\wte_{n+k+1}, \quad k=2,3,\ldots.
\end{split}
\end{equation}
We want to find an orthogonal matrix $B=\|b_{ij}\|_{i,j=1}^{n-1}$ (more precisely,
the matrix with orthogonal columns) such that
\begin{equation}\label{newbas}
\wte_k=\sum_{i=1}^{n-1}b_{ik}e_i, \qquad k=1,2,\ldots,n-1.
\end{equation}
The following notation will be convenient throughout the proof
\begin{equation}
w_{jk}:=\sum_{i=1}^{n-1} b_{ik}a_{ij}, \qquad j=1,2,\ldots,n, \quad k=1,2,\ldots,n-1.
\end{equation}

1. Put
$$ b_{i,n-1}:=a_{in}, \quad i=1,2,\ldots,n-1, \qquad \wte_{n-1}:=\sum_{i=1}^{n-1}a_{in}e_i, $$
and so
\begin{equation}
A\wte_n=Ae_n=\wte_{n-1}+a_{nn}\wte_n+d\wte_{n+1}.
\end{equation}
Next,
\begin{equation*}
\begin{split}
A\wte_{n-1} &= \sum_{i=1}^{n-1}a_{in}Ae_i=\sum_{i=1}^{n-1}a_{in}\sum_{j=1}^{n}a_{ij}e_j=
\sum_{j=1}^{n}u_{j,n-1}e_j \\
 &= \sum_{j=1}^{n-1} u_{j,n-1}e_j+\sum_{i=1}^{n-1}a_{in}^2\cdot \wte_n.
\end{split}
\end{equation*}
Note that $\sum_{i}a_{in}^2=\sum_i b_{i,n-1}^2>0$, for otherwise the original matrix already
has the canonical form. Define
\begin{equation}\label{step1}
y_1:=\frac{\sum_{j=1}^{n-1} w_{j,n-1}b_{j,n-1}}{\sum_{j=1}^{n-1} b_{j,n-1}^2}\,, \qquad b_{j,n-2}:=w_{j,n-1}-y_1b_{j,n-1}.
\end{equation}
Then
$$
A\wte_{n-1}=\wte_{n-2}+y_1\wte_{n-1}+z_1\wte_n, \quad z_1:=\sum_{i=1}^{n-1}b_{i,n-1}^2,
$$
and orthogonality of the columns
\begin{equation}\label{orth1}
\sum_{i=1}^{n-1} b_{i,n-2}b_{i,n-1}=0
\end{equation}
follows directly from the choice of $y_1$ \eqref{step1}.

If $\wte_{n-2}=0$, that is, $b_{i,n-2}=0$ for all $i=1,\ldots,n-1$, the algorithm terminates,
since
$$
A\wte_{n-1}=y_1\wte_{n-1}+z_1\wte_n,
$$
the subspace $\cl_{n-1}:=\lc\{\wte_j\}_{j\ge n-1}$ is invariant for $A$ and so is its orthogonal
complement $\cl_{n-1}^{\bot}=\ell^2\ominus\cl_{n-1}$, which is finite dimensional. Thus, we can
extend the basis $\{\wte_j\}_{j\ge n-1}$ in $\cl_{n-1}$ in an arbitrary way to the basis in the whole
$\ell^2$. The latter is exactly the basis we are looking for.

If $\wte_{n-2}\not=0$ we proceed to the next step.

2. We have
\begin{equation*}
\begin{split}
A\wte_{n-2} &= \sum_{i=1}^{n-1}b_{i,n-2}Ae_i=\sum_{i=1}^{n-1}b_{i,n-2}\sum_{j=1}^{n}a_{ij}e_j=
\sum_{j=1}^{n}w_{j,n-2}e_j \\
 &= \sum_{j=1}^{n-1} w_{j,n-1}e_j+w_{n,n-2}e_n=\sum_{j=1}^{n-1} w_{j,n-1}e_j
\end{split}
\end{equation*}
($w_{n,n-2}=0$ by \eqref{orth1}). Put
\begin{equation}\label{step2}
y_2:=\frac{\sum_{j=1}^{n-1} w_{j,n-2}b_{j,n-2}}{\sum_{j=1}^{n-1} b_{j,n-2}^2}\,, \quad
z_2:=\frac{\sum_{j=1}^{n-1} w_{j,n-2}b_{j,n-1}}{\sum_{j=1}^{n-1} b_{j,n-1}^2}
\end{equation}
and define
\begin{equation}\label{step2.1}
b_{j,n-3}:=w_{j,n-2}-y_2b_{j,n-2}-z_2b_{j,n-1}, \qquad j=1,2,\ldots,n-1,
\end{equation}
which leads to
\begin{equation}\label{step2.2}
A\wte_{n-2}=\wte_{n-3}+y_2\wte_{n-2}+z_2\wte_{n-1}.
\end{equation}
The orthogonality relations
\begin{equation}\label{orth2}
\sum_{i=1}^{n-1} b_{i,n-3}b_{i,n-2}=\sum_{i=1}^{n-1} b_{i,n-3}b_{i,n-1}=0
\end{equation}
stem from the definition of $y_2$,$z_2$ and \eqref{orth1}.

Again, if $\wte_{n-3}=0$, we are done. If not, we proceed to the next step.

3. As above we have in view of \eqref{orth2}
\begin{equation*}
A\wte_{n-3} = \sum_{j=1}^{n}w_{j,n-3}e_j = \sum_{j=1}^{n-1} w_{j,n-3}e_j.
\end{equation*}
Put
\begin{equation}\label{step3}
y_3:=\frac{\sum_{j=1}^{n-1} w_{j,n-3}b_{j,n-3}}{\sum_{j=1}^{n-1} b_{j,n-3}^2}\,, \quad
z_3:=\frac{\sum_{j=1}^{n-1} w_{j,n-3}b_{j,n-2}}{\sum_{j=1}^{n-1} b_{j,n-2}^2}
\end{equation}
and define
\begin{equation}\label{step2.1}
b_{j,n-4}:=w_{j,n-2}-y_2b_{j,n-2}-z_2b_{j,n-1}, \qquad j=1,2,\ldots,n-1,
\end{equation}
so
\begin{equation}\label{step2.2}
A\wte_{n-3}=\wte_{n-4}+y_3\wte_{n-3}+z_3\wte_{n-2}.
\end{equation}
As above, the orthogonality relations
\begin{equation}\label{orth2}
\sum_{i=1}^{n-1} b_{i,n-4}b_{i,n-3}=\sum_{i=1}^{n-1} b_{i,n-4}b_{i,n-2}=0
\end{equation}
stem from the definition of $y_3$,$z_3$ and \eqref{orth2}. As for the last one
we see from \eqref{step2.1} and \eqref{orth2} that
\begin{equation*}
\begin{split}
\sum_{i=1}^{n-1} b_{i,n-4}b_{i,n-1} &=\sum_{i=1}^{n-1} w_{i,n-3}b_{i,n-1}=
\sum_{i=1}^{n-1}\left(\sum_{j=1}^{n-1}b_{j,n-3}a_{ji}\right)b_{i,n-1} \\
&=\sum_{j=1}^{n-1}b_{j,n-3}\left(\sum_{i=1}^{n-1}b_{i,n-1}a_{ij}\right)=
\sum_{j=1}^{n-1}b_{j,n-3}w_{j,n-1}=0
\end{split}
\end{equation*}
since $w_{j,n-1}=b_{j,n-2}+y_1b_{j,n-1}$ (see \eqref{step1}).

k. As soon as the process has not terminated, the pairwise orthogonal columns
$\{b_{i,n-1}\}_{i=1}^{n-1}, \ldots \{b_{i,n-k}\}_{i=1}^{n-1}$ have been constructed,
\begin{equation}\label{orthk}
\sum_{i=1}^{n-1} b_{i,n-p}b_{i,n-q}=0, \qquad p\not=q, \quad p,q=1,2,\ldots,k,
\end{equation}
and the equalities
\begin{equation}\label{stepk}
b_{i,n-m}+y_{m-1}b_{i,n-m+1}+z_{m-1}b_{i,n-m+2}=w_{i,n-m+1}, \qquad m=2,\ldots,k, \quad z_1=0
\end{equation}
hold. Next, we have
\begin{equation*}
A\wte_{n-k} = \sum_{j=1}^{n}w_{j,n-k}e_j = \sum_{j=1}^{n-1} w_{j,n-k}e_j,
\end{equation*}
as $w_{n,n-k}=0$ in view of \eqref{orthk}. Again put
\begin{equation}\label{stepk.1}
y_k:=\frac{\sum_{j=1}^{n-1} w_{j,n-k}b_{j,n-k}}{\sum_{j=1}^{n-1} b_{j,n-k}^2}\,, \quad
z_k:=\frac{\sum_{j=1}^{n-1} w_{j,n-k}b_{j,n-k+1}}{\sum_{j=1}^{n-1} b_{j,n-k+1}^2}\,,
\end{equation}
and define
\begin{equation}\label{stepk.2}
b_{j,n-k-1}:=w_{j,n-k}-y_kb_{j,n-k}-z_kb_{j,n-k+1}, \qquad j=1,2,\ldots,n-1,
\end{equation}
so
\begin{equation}\label{step2.2}
A\wte_{n-k}=\wte_{n-k-1}+y_k\wte_{n-k}+z_k\wte_{n-k+1}.
\end{equation}

It remains only to verify the orthogonality conditions with the last column $\{b_{i,n-k-1}\}$.
The first two
\begin{equation*}
\sum_{i=1}^{n-1} b_{i,n-k-1}b_{i,n-k}=\sum_{i=1}^{n-1} b_{i,n-k-1}b_{i,n-k+1}=0
\end{equation*}
follow directly from the choice of $y_k,z_k$. For $m=1,2,\ldots,k-2$
\begin{equation*}
\begin{split}
\sum_{i=1}^{n-1} b_{i,n-k-1}b_{i,n-m} &=\sum_{i=1}^{n-1} w_{i,n-k}b_{i,n-m}=
\sum_{i=1}^{n-1}\left(\sum_{j=1}^{n-1}b_{j,n-k}a_{ji}\right)b_{i,n-m} \\
&=\sum_{j=1}^{n-1}b_{j,n-k}\left(\sum_{i=1}^{n-1}b_{i,n-m}a_{ij}\right)=
\sum_{j=1}^{n-1}b_{j,n-k}w_{j,n-m} \\
& = \sum_{j=1}^{n-1}b_{j,n-k}(b_{j,n-m-1}+y_{m-2}b_{j,n-m}+z_{m-2}b_{j,n-m+1})=0,
\end{split}
\end{equation*}
as claimed.

If $\wte_{n-k-1}=0$ then $A\wte_{n-k}=y_k\wte_{n-k}+z_k\wte_{n-k+1}$,
$\cl_{n-k}:=\lc\{\wte_j\}_{j\ge n-k}$ is $A$-invariant, and so is its orthogonal complement
which is finite dimensional. Otherwise the algorithm can be extended to yet another step.

When the algorithm continues till $\wte_1$, we come to the orthogonal set $\{\wte_j\}_{j=1}^{n-1}$
in $\lc\{e_j\}_{1\le j\le n-1}$, so necessarily $\wte_0=0$. In this case the finite-dimensional
component is missing. The operator $A$ in the normalized basis $\{\whe_k\}_{k\ge1}$,
$\whe_k=\|\wte_k\|^{-1}\,\wte_k$ is given by the Jacobi matrix which agrees with the original matrix $J$
from some point on. The proof is complete.

\section{Spectra of certain trees with tails}
\label{s2}

As we have already mentioned in Introduction the adjacency matrix $A(\Gamma)$ for the coupling
$\Gamma=G+\bp_\infty(\{a_j\})$ is \eqref{adjmat}. The algorithm suggested in Theorem \ref{algorithm}
provides a way to implement the spectral analysis of the adjacency operator $A(\Gamma)$ which
consists of two stages. First, one has to apply the above algorithm by hand to obtain the canonical form \eqref{canform} and to determine the finite-dimensional and Jacobi components of $\gg$.
Second, the solutions of two algebraic equations (the characteristic equation for $F(\gg)$ and the Jost
equation for $J(\gg)$) provide the spectrum of $\gg$.

\begin{example}\label{starsimp} ``A weighted star''.

We begin with the coupling $\Gamma=S_n(w)+\bp_\infty$, where $S_n(w)$ is a simple weighted star
graph of order $n+1$, $n\ge2$, with vertices $1,\ldots,n$ of degree $1$ and the weight of the edge
$(k,n+1)$ is $w_k$, $1\le k\le n$. The canonical basis $\{\whe_k\}_{k\in\bn}$ looks as follows.
We put
$$ \whe_j:=e_j, \quad j\ge n+1 \ \Longrightarrow \ A(\Gamma)\,
\whe_j=\whe_{j-1}+\whe_{j+1}, \quad j\ge n+2. $$
Next, let $w:=(w_1,w_2,\ldots,w_n)$, $\|w\|=\sqrt{w_1^2+\ldots+w_n^2}$, and let
$$ \whe_{n}:=\frac1{\|w\|}\,\sum_{j=1}^{n} w_k\,e_j\,. $$
Then
$$ A(\Gamma)\,\whe_{n+1}=\|w\|\,\whe_{n}+\whe_{n+2}, \qquad A(\Gamma)\,\whe_{n}=\|w\|\,\whe_{n+1}. $$
So the Jacobi subspace and Jacobi component of $\gg$ are
\begin{equation}\label{jaccom1}
\cl_J=\lc\{\whe_j\}_{j\ge n}, \quad J(\gg)=J\bigl(\{0\},\{\|w\|,1,1,\ldots\}\bigr).
\end{equation}

To find the finite-dimensional component we construct an orthonormal basis in $\bc^n$ by means of
a unitary matrix $\xi=\|\xi_{kj}\|_{k,j=1}^n$ with the specified last column
\begin{equation}\label{ortbas}
f_j:=\sum_{k=1}^n \xi_{kj}e_k(n), \qquad \xi_{kn}=\frac{w_k}{\|w\|}\,, \quad k=1,\ldots,n,
\end{equation}
where $\{e_k(n)\}_{k=1}^n$ is the standard basis in $\bc^n$. Put
\begin{equation}\label{star11}
 \whe_j:=\{f_j,0,0,\ldots\}, \quad j=1,\ldots,n.
\end{equation}
The orthogonality relations $\langle \whe_k,\whe_{n}\rangle=0$, $1\le k\le n-1$, give
\begin{equation}\label{ortrel}
A(\gg)\,\whe_k= \sum_{j=1}^{n} \xi_{kj}\xi_{kn}\cdot \whe_{n+1}=0.
\end{equation}
Hence the finite-dimensional component $F(\gg)=\mathbb O_{n-1}$ on the subspace
$\lc\{\whe_j\}_{j=1}^{n-1}$. So the canonical form is
\begin{equation}\label{canfor1}
\ag\simeq \bo_{n-1}\bigoplus J\bigl(\{0\},\{\|w\|,1,1,\ldots\}\bigr).
\end{equation}

The Jost polynomial is now given by \eqref{jf1}
$$ \|w\|\,u(z)=(1-\|w\|^2)z^2+1. $$
Clearly, $u>0$ for $\|w\|\le1$, and it has zeros inside $(-1,1)$ if and only if $\|w\|>\sqrt2$.
In this case the discrete spectrum is
\begin{equation}\label{starweight}
\s_d(S_n(w)+\bp_\infty) =\Bigl\{0^{(n-1)},
\pm\Bigl(\sqrt{\|w\|^2-1}+\frac1{\sqrt{\|w\|^2-1}}\Bigr)\Bigr\}.
\end{equation}

For the unweighted star $S_n$ the discrete spectrum is
\begin{equation}\label{spec1}
\begin{split}
\s_d(S_n+\bp_\infty) &=\Bigl\{0^{(n-1)},
\pm\Bigl(\sqrt{n-1}+\frac1{\sqrt{n-1}}\Bigr)\Bigr\},
\quad n\ge3, \\ \s_d(S_2+\bp_\infty) &=\{0^{(1)}\}.
\end{split}
\end{equation}
This case is studied in \cite{LeNi-umzh}.

Note that $S_n$ is a complete bipartite graph, $S_n=K_{1,n}$. For the general complete bipartite
graph $K_{p,n+1-p}$ see Example \ref{bipartite} below.
\end{example}

\begin{remark}\label{fourier}
Although the explicit form of the matrix $\xi=\|\xi_{kj}\|_1^n$ in \eqref{ortbas} is immaterial,
it is worth noting that in the unweighted case  $\xi_{kn}=n^{-1/2}$, $1\le k\le n$, and one can take
\begin{equation}\label{fourmat}
\xi=\cf_n:=\frac1{\sqrt{n}}\,\|\eps_n^{kj}\|_{k,j=1}^n, \qquad \eps_n:=e^{\frac{2\pi i}{n}}\,,
\end{equation}
which is known as the {\it Fourier matrix}. Clearly, there are lots of options for $\xi$ to be a
real orthogonal matrix (rotation in $\br^n$ with appropriate Euler's angles, orthogonal polynomials etc.).
\end{remark}

\begin{example}\label{starmult} ``A multiple star''.

Consider an unweighted star-like graph $S_{n,p}$ with $n$ rays, $n\ge2$, each of which contains
$p+1$ vertices, $p\ge2$. The vertices are numbered as
$$ \{1,n+1,\ldots,(p-1)n+1\},\ \ \{2,n+2,\ldots,(p-1)n+2\},\ \ \ldots \ \ \{n,2n,\ldots,pn\}, $$
and the root is $pn+1$, so $S_{n,1}=S_n$. Let $\gg=S_{n,p}+\bp_\infty$.

As above, we put $\whe_j:=e_j$, $j=pn+1,\ldots$, and
\begin{equation}
\whe_{p(n-1)+i}:=\frac1{\sqrt{n}}\,\sum_{q=1}^{n} e_{(i-1)n+q}, \qquad i=1,2,\ldots,p.
\end{equation}
We have $A(\gg)\,\whe_{j}=\whe_{j-1}+\whe_{j+1}$, $j=pn+2\ldots$,
\begin{equation*}
\begin{split}
A(\gg)\whe_{pn+1} &=\sqrt{n}\whe_{pn}+\whe_{pn+2}, \\
A(\gg)\,\whe_{pn} &=\whe_{pn-1}+\sqrt{n}\,\whe_{pn+1}, \\
A(\gg)\,\whe_{p(n-1)+i} &=\whe_{p(n-1)+i-1}+\whe_{p(n-1)+i+1}, \quad i=2,\ldots,p-1, \\
A(\gg)\,\whe_{p(n-1)+1} &=\whe_{p(n-1)+2},
\end{split}
\end{equation*}
so the Jacobi subspace and Jacobi component of $\gg$ are
\begin{equation}\label{jaccom2}
\begin{split}
\cl_J &=\lc\{\whe_j\}_{j\ge p(n-1)+1}, \\ J(\gg) &=J(\{0\},\{a_j\}), \quad
a_j=
\left\{
  \begin{array}{ll}
    \sqrt{n}, & j=p; \\
    1, & j\not=p.
  \end{array}
\right.
\end{split}
\end{equation}

To find the finite-dimensional component note that, by the construction,
$\whe_{p(n-1)+k}\in\lc\{e_{(k-1)n+1},\ldots, e_{kn}\}$. As in the above example,
we supplement each $\whe_{p(n-1)+i}$ to the basis in this subspace by
means of the Fourier matrix \eqref{fourmat}
$$ f_j^{(k)}:=\sum_{q=1}^n \xi_{qj}e_{(k-1)n+q}, \quad f_n^{(k)}:=\sum_{q=1}^n \xi_{qn}e_{(k-1)n+q}
=\whe_{p(n-1)+k} $$
for $1\le j\le n-1$, $1\le k\le p$. As in \eqref{ortrel} we have
\begin{equation}
\begin{split}\label{adjact}
A(\gg)f_j^{(1)} &=f_j^{(2)}, \quad A(\gg)f_j^{(2)}=f_j^{(1)}+f_j^{(3)}, \ldots, \\
A(\gg)f_j^{(p-1)} &=f_j^{(p-2)}+f_j^{(p)}, \quad A(\gg)f_j^{(p)}=f_j^{(p-1)}.
\end{split}
\end{equation}
Relations \eqref{adjact} mean that the subspace $\ch_j:=\lc\{f_j^{(1)}, \ldots, f_j^{(p)}\}$ is
$A(\gg)$-invariant, and $A(\gg)\vert \ch_j=A(\bp_p)$. There are exactly $n-1$ such subspaces
for $j=1,\ldots,n-1$. Finally, we come to the following canonical form for the adjacency matrix
\begin{equation}\label{canfor2}
A(\gg)\simeq \Bigl(\bigoplus_{j=1}^{n-1}A(\bp_p)\Bigr) \bigoplus J(\gg).
\end{equation}

The Jost polynomial is computed in \eqref{rank2}
$$ -\sqrt{n}u(z)=(n-1)z^2\,\frac{z^{2p}-1}{z^2-1}-1=\frac{(n-1)z^{2p+2}-nz^2+1}{z^2-1}\,. $$
It is easy to see that the polynomial $q(x)=(n-1)x^{p+1}-nx+1$ has exactly two positive roots
$0<x_1(p,n)<x_1'(p,n)=1$, so the first one has the spectral meaning. Hence
\begin{equation}\label{spec2}
\s_d(\gg)=\Bigl\{2\cos\frac{\pi j}{p+1}\Bigr\}_{j=1}^p
\bigcup\,\Bigl\{\pm\Bigl(\sqrt{x_1(p,n)}+\frac1{\sqrt{x_1(p,n)}}\Bigr)\Bigr\}\,,
\end{equation}
so we have $p$ eigenvalues on the absolutely continuous spectrum $[-2,2]$ of multiplicity $n-1$ and
two simple eigenvalues off $[-2,2]$.

Note that the weighted multiple star can be treated in exactly the same fashion.

The spectrum of the multiple star was studied in \cite{LeNi-dan}, but no explicit formulae were provided.

\end{example}

\medskip

The problem becomes harder (in the sense of computation) if the original finite star-like graph is
nonsymmetric (the rays are different).

\begin{example}\label{sword} ``A sword''.

\begin{picture}(300, 160)
\multiput(60,80) (30,0) {5} {\circle* {4}}
\multiput(90,140) (0,-30) {5} {\circle* {4}}
\multiput(62,80) (30,0) {4} {\line(1, 0) {26}}
\multiput(90,138) (0,-30) {4} {\line(0,-1) {26}}
\multiput(200,80) (10,0) {3} {\circle* {2}}
\put(60,84) {$1$}
\put(94,84) {$6$}
\put(94,110) {$3$}
\put(94,140) {$2$}
\put(94,50) {$5$}
\put(94,20) {$4$}
\put(122,83) {$7$}
\put(152,83) {$8$}
\put(182,83) {$9$}
\put(140,50) {$T_{1,2,2,\infty}$}

\end{picture}

This graph (also known as $T(1,2,2,\infty)$) can be viewed as a coupling $\gg=T(1,2,2)+\bp_\infty$.
We put $\whe_k=e_k$, $k\ge6$, so
$$ A(\gg)\,\whe_k=\whe_{k-1}+\whe_{k+1}, \qquad k\ge7. $$
Next, let
$$ \whe_5:=\frac{e_1+e_3+e_5}{\sqrt{3}}\,, \qquad \whe_4:=\frac{e_2+e_4}{\sqrt{2}}\,, $$
so that
$$ A(\gg)\,\whe_6=\sqrt{3}\,\whe_5+\whe_7, \qquad A(\gg)\,\whe_5=\sqrt{\frac23}\,
\whe_4+\sqrt{3}\,\whe_6. $$
We want to determine the vector $\whe_3$ from the equation
$$ A(\gg)\,\whe_4=a\whe_3+\sqrt{\frac23}\,\whe_5=a\whe_3+\frac{\sqrt{2}}3\,(e_1+e_3+e_5), $$
which gives
$$ a=\frac1{\sqrt{3}}\,, \qquad \whe_3:=\frac{-2e_1+e_3+e_5}{\sqrt{6}}\,, $$
and so
$$ A(\gg)\,\whe_3=\frac1{\sqrt{6}}\,(-2e_6+e_6+e-2+e_6+e-4)=\frac{e_2+e_4}{\sqrt{6}}=
\frac{\whe_4}{\sqrt{3}}\,. $$
Clearly, $\whe_3$ is orthogonal to $\whe_j$ for $j\ge4$, and $\|\whe_3\|=1$.
We come thereby to the Jacobi subspace and Jacobi component of $\gg$
\begin{equation}\label{jaccom3}
\cl_J=\lc\{\whe_j\}_{j\ge3}, \qquad J(\gg)=
J\Bigl(\{0\}, \Bigl\{\frac1{\sqrt{3}}, \sqrt{\frac23}, \sqrt{3}, 1,1\ldots \Bigr\}\Bigr).
\end{equation}

To compute the finite-dimensional component, we put
$$ \whe_2:=\frac{e_3-e_5}{\sqrt{2}}\,, \qquad \whe_1:=\frac{e_2-e_4}{\sqrt{2}}\,, $$
so $A(\gg)\,\whe_2=\whe_1$, $A(\gg)\,\whe_1=\whe_2$, and $\{\whe_j\}_{j\ge1}$ turms out to
be the canonical basis in $\ell^2$ for $A(\gg)$. The adjacency operator is unitarily equivalent to
\begin{equation}\label{canfor3}
A(\gg)\simeq F(\gg)\bigoplus J(\gg),
\qquad F(\gg)=A(\bp_2)=\begin{bmatrix}
0 & 1 \\
1 & 0 &
\end{bmatrix}
\end{equation}

The Jost polynomial can be computed directly from relations \eqref{3term}
$$ u(z)=\sqrt{\frac32}\,(1-z^2-3z^4-2z^6). $$
It is not hard to see that the cubic polynomial $p(x)=2x^3+3x^2+x-1$ has the only real root $x_1$,
$0<x_1<1$. Hence
\begin{equation}\label{spec3}
\s_d(\gg)=\Bigl\{\pm1, \pm\Bigl(\sqrt{x_1}+\frac1{\sqrt{x_1}}\Bigr)\Bigr\}.
\end{equation}
\end{example}

\section{Couplings of graphs with cycles and their spectra}

\begin{example}\label{kite} ``A kite''.

\begin{picture}(300,120)
\put(60,60) {\circle*{4}} \put(80,25) {\circle*{4}} \put(120,25) {\circle*{4}}
\put(140,60) {\circle*{4}} \put(120,95) {\circle*{4}} \put(80,95) {\circle*{4}}
\multiput(180,60)(40,0){2} {\circle*{4}}  \multiput(235,60)(10,0){3} {\circle*{2}}
\multiput(142,60)(40,0){2} {\line(1,0){36}}
\put(81,95) {\line(1,0){40}} \multiput(90,25)(10,0){3} {\circle*{2}} \put(100,50) {$\mathbb{C}_{m}$}
\put(120,25) {\line(3,5){20}} \put(60,60) {\line(3,5){20}}
\put(120,95) {\line(3,-5){20}}  \put(60,60) {\line(3,-5){20}}
\put(143,64) {$m$} \put(167,64) {$m+1$} \put(205,64) {$m+2$}
\put(122,98) {$1$} \put(74,98) {$2$} \put(52,58) {$3$} \put(70,23) {$4$} \put(128,25) {$m-1$}
\end{picture}

Consider the coupling of $m$-cycle and the infinite path $\gg=\bc_m+\bp_\infty$. The construction
depends heavily on the parity of $m$.

{\bf Case 1}. Let $m=2n$, $n\ge2$. We put as usual $\whe_k:=e_k$, $k\ge 2n$, and
\begin{equation*}
\whe_j:=\frac{e_{j}+e_{2n-j}}{\sqrt{2}}\,, \quad j=n+1,\ldots,2n-1, \qquad \whe_n=e_n.
\end{equation*}
Then
$$
\left\{
  \begin{array}{ll}
    \ag\whe_{2n}=\sqrt{2}\,\whe_{2n-1}+\whe_{2n+1}, \\
    \ag\whe_{2n-1}=\whe_{2n-2}+\sqrt{2}\,\whe_{2n},
  \end{array}
\right. \qquad A(\gg)\,\whe_{n+j}=\whe_{n+j-1}+\whe_{n+j+1}
$$
for $2\le j\le n-2$, and
$$ A(\gg)\,\whe_{n+1}=\sqrt{2}\,\whe_n+\whe_{n+2}, \qquad A(\gg)\,\whe_n=\sqrt{2}\,\whe_{n+1}. $$
The Jacobi subspace and component are
\begin{equation}\label{jaccom4}
\cl_J=\lc\{\whe_j\}_{j\ge n}, \ \ J(\gg):=J(\{0\},\{a_j\}), \ \
a_j=
\left\{
  \begin{array}{ll}
    1, & j\not= 1,n, \\
    \sqrt{2}, & j=1,n.
  \end{array}
\right.
\end{equation}

Next, put
$$ \whe_j:=\frac{e_j-e_{2n-j}}{\sqrt{2}}\,, \qquad j=1,\ldots,n-1. $$
Then $\{\whe_j\}_{j\ge1}$ is the orthonormal basis in $\ell^2$ and
$$ \ag\whe_{n-1}=\whe_{n-2}, \quad \ag\whe_{j}=\whe_{j-1}+\whe_{j+1}, \ \ j=2,\ldots,n-2,
\quad \ag\whe_1=\whe_2, $$
so finally
\begin{equation}\label{canfor4}
A(\gg)\simeq F(\gg)\bigoplus J(\gg),
\qquad F(\gg):=A(\bp_{n-1}).
\end{equation}

The Jost polynomial is computed in \eqref{jf5}
$$ -u(z)=z^{2n}+2z^2-1, $$
so, by the Descarte's rule, it has two real roots $-1<x_2(n)<0<x_1(n)<1$, $x_1(n)=-x_2(n)$.
Hence
\begin{equation}\label{spec4}
\s_d(\bc_{2n}+\bp_\infty)=
\Bigl\{2\cos\frac{\pi j}{n}\Bigr\}_{j=1}^{n-1}\bigcup
\Bigl\{\pm\Bigl(x_1(n)+\frac1{x_1(n)}\Bigr)\Bigr\}.
\end{equation}

{\bf Case 2}. Let $m=2n+1$, $n\ge1$. With $\whe_k:=e_k$, $k\ge 2n+1$, we put
$$ \whe_{n+j}:=\frac{e_{n-j+1}+e_{n+j}}{\sqrt{2}}\,, \quad j=1,2,\ldots,n, $$
so
\begin{equation*}
\begin{split}
\ag\whe_j &=\whe_{j-1}+e_{j+1}, \quad j=2n+2,2n+3,\ldots \\
\ag\whe_{2n+1} &=\sqrt{2}\,\whe_{2n}+\whe_{2n+2}, \quad \ag\whe_{2n} =
\whe_{2n-1}+\sqrt{2}\,\whe_{2n+1}, \\
\ag\whe_j &=\whe_{j-1}+e_{j+1}, \quad j=n+2,\ldots,2n-1, \\
\ag\whe_{n+1} &=\whe_{n+1}+\whe_{n+2}
\end{split}
\end{equation*}
(with the obvious modification for $n=1,2$). The Jacobi subspace and Jacobi components
are now $\cl_J=\lc\{\whe_j\}_{j\ge n+1}$,
\begin{equation}\label{jaccom5}
J(\gg):=J(\{1,0,0,\ldots\},\{a_j\}), \quad
a_j=
\left\{
  \begin{array}{ll}
    1, & j\not= n; \\
    \sqrt{2}, & j=n.
  \end{array}
\right.
\end{equation}

The finite-dimensional component arises from the complement to the basis in the entire space
$$ \whe_j:=\frac{e_{n-j+1}-e_{n+j}}{\sqrt{2}}\,, \quad j=1,2,\ldots,n $$
and relations
$$ \ag\whe_1=-\whe_1+\whe_2, \quad \ag\whe_j =\whe_{j-1}+e_{j+1}, \ \ j=2,\ldots,n-1, \ \
\ag\whe_n=\whe_{n-1}, $$
so finally
\begin{equation}\label{canfor5}
A(\gg)\simeq F(\gg)\bigoplus J(\gg),
\qquad F(\gg):=J(\{-1,0,\ldots,0\}, \{1\})
\end{equation}
is a Jacobi matrix of order $n$.

The Jost polynomial is given in \eqref{jf3}
$$ -\sqrt2(z+1)u(z)=z^{2n+1}+2z^2-1. $$
The Descarte's rule applied to the polynomial $p_{2n+1}(x)=x^{2n+1}+2x^2-1$ shows that there
are two real roots such that $ -1<x_2(n)<0<x_1(n)<1$. Both of them contribute to the discrete
spectrum of the Jacobi component
$$ \s_d(J(\gg))=\Bigl\{x_1(n)+\frac1{x_1(n)}\,, x_2(n)+\frac1{x_2(n)}\Bigr\}\,. $$

To find the spectrum of $F(\gg)$ in \eqref{canfor5} we expand the characteristic
determinant over the first row
$$ \det|F(\gg)+x|=\det|A(\bp_n)+x|-\det|A(\bp_{n-1}+x|=U_n\Bigl(\frac{x}2\Bigr)-
U_{n-1}\Bigl(\frac{x}2\Bigr), $$
where $U_n$ is the Chebyshev polynomial of the second kind,
$$ U_n(\cos t):=\frac{\sin(n+1)t}{\sin t}\,, \quad U_n(\cos t)-U_{n-1}(\cos t)=
\frac{\cos\frac{2n+1}2\,t}{\cos\frac{t}2}\,. $$
So the spectrum of $F(\gg)$ is
\begin{equation}\label{fincomkite}
\s(F(\gg))=\Bigl\{2\cos\frac{2\pi j}{2n+1}\Bigr\}_{j=1}^{n}.
\end{equation}
\end{example}

Actually, we can combine the latter two results in a unique way.

\begin{proposition}\label{speckite}
The spectrum of the graph $\gg=\bc_m+\bp_\infty$ is
$$ \s(\bc_m+\bp_\infty)=[-2,2]\cup\s_d(\bc_m+\bp_\infty) $$
with the discrete component
\begin{equation}\label{spectrkite}
\s_d(\bc_m+\bp_\infty) =\Bigl\{2\cos\frac{2\pi k}{m}\Bigr\}_{k=1}^l\bigcup\,\{\l_1(m),\l_2(m)\},
 \ l=\Bigl[\frac{m-1}2\Bigr],
\end{equation}
where
$$ \l_j(m) =x_j(m)+\frac1{x_j(m)}\,, \quad j=1,2, \quad -1<x_2(m)<0<x_1(m)<1 $$
are real roots of the polynomial $p(x)=x^m+2x^2-1$.
\end{proposition}

\begin{example}\label{bipartite} ``The complete bipartite graph''.

Let $G=K_{p,n+1-p}$ be the complete bipartite graph of order $n+1$ \cite[p.16]{CDS80},
$\gg=K_{p,n+1-p}+\bp_\infty$. Define
$$ S_1:=e_1+\ldots+e_p, \qquad S_2:=e_{p+1}+\ldots+e_n, $$
and put
\begin{equation*}
\whe_j :=e_j, \quad j=n+1, n+2,\ldots, \quad \whe_n :=\frac{S_1}{\sqrt{p}}\,, \quad \whe_{n-1}:=\frac{S_2}{\sqrt{n-p}}\,.
\end{equation*}
Then $\ag\whe_{j}=\whe_{j-1}+\whe_{j+1}$ for $j\ge n+2$ and
$$ \ag\whe_{n+1}=S_1+\whe_{n+2}=\sqrt{p}\,\whe_{n}+\whe_{n+2}. $$
For the bipartite graph we have $\ag e_k=S_2+e_{n+1}$ for $1\le k\le p$, and so
$$ \ag\whe_{n}=\frac1{\sqrt{p}}\,p(S_2+\whe_{n+1})=
\sqrt{p(n-p)}\,\whe_{n-1}+\sqrt{p}\,\whe_{n+1}. $$
Next, since $\ag e_k=S_1$ for $p+1\le k\le n$, then $\ag\whe_{n-1}=\sqrt{p(n-p)}\,\whe_n$,
and the Jacobi component is
\begin{equation}\label{jaccom7}
\cl_J=\lc\{\whe_j\}_{j\ge n-1}, \quad
J(\gg):=J(\{0\},\{\sqrt{p(n-p)},\sqrt{p},1,1,\ldots\}).
\end{equation}

To find the finite-dimensional component, note that $\whe_n\in\lc\{e_k\}_{k=1}^p$, $\whe_{n-1}\in\lc\{e_k\}_{k=p+1}^n$,
so we amplify these two vectors to the orthonormal base
in the corresponding subspaces by means of the Fourier matrices $\xi=\cf_p$ and $\eta=\cf_{n-p}$
\begin{equation*}
\begin{split}
\whe_j &:=\sum_{k=1}^p \xi_{kj}e_k, \quad j=1,\ldots,p-1, \quad \quad \ \ \
\whe_n=\frac1{\sqrt{p}}\,\sum_{k=1}^p e_k,   \\
\whe_j &:=\sum_{k=p+1}^n \eta_{kj}e_k, \quad j=p,\ldots,n-2, \quad
\whe_{n-1}=\frac1{\sqrt{n-p}}\,\sum_{k=p+1}^n e_k.
\end{split}
\end{equation*}
We have $\ag\whe_k=0$ for $1\le k\le n-2$, so
\begin{equation}\label{canfor7}
A(\gg)\simeq F(\gg)\bigoplus J(\gg),
\qquad F(\gg):=\bo_{n-2}.
\end{equation}

The Jost polynomial is given in \eqref{jf2}
$$ -p\sqrt{n-p}\,u(z)=(p-1)z^4+(p(n-p)+p-2) z^2-1=q(z^2). $$
It is easy to see that the quadratic polynomial $q$ has the only root $x_1(p,n)$
in $(-1,1)$, $0<x_1(p,n)<1$, so
\begin{equation}\label{spec7}
\s_d(K_{p,n+1-p}+\bp_\infty)=\{0^{(n-2)}\}\bigcup\,
\Bigl\{\pm\Bigl(\sqrt{x_1(p,n)}+\frac1{\sqrt{x_1(p,n)}}\Bigr)\Bigr\}.
\end{equation}

For $p=1$ we have Example \ref{starsimp}.
\end{example}

\begin{example}\label{propeller} ``A propeller with equal blades''.

\begin{picture}(300,180)

\multiput(40,80)(30,0){4}  {\circle*{4}} \put(160,110) {\circle*{4}}
\multiput(40,140)(30,0){4}  {\circle*{4}}
\multiput(190,80)(30,0){4} {\circle*{4}}
\multiput(190,140)(30,0){4} {\circle*{4}}
\put(40,82) {\line(0,1){56}} \put(280,82) {\line(0,1){56}}

\put(192,80) {\line(1,0){26}} \put(252,80) {\line(1,0){26}}
\put(42,80) {\line(1,0){26}}  \put(102,80) {\line(1,0){26}}
\put(192,140) {\line(1,0){26}} \put(252,140) {\line(1,0){26}}
\put(42,140) {\line(1,0){26}}  \put(102,140) {\line(1,0){26}}

\multiput(76,80)(8,0){3} {\circle*{2}} \multiput(226,80)(8,0){3} {\circle*{2}}
\multiput(76,140)(8,0){3} {\circle*{2}} \multiput(226,140)(8,0){3} {\circle*{2}}

\put(130,80) {\line(1,1){30}} \put(160,110) {\line(1,1){30}}
\put(130,140) {\line(1,-1){30}} \put(160,110) {\line(1,-1){30}}

\multiput(160,110)(0,-30) {3} {\circle* {4}}
\multiput(160,42)(0,-8) {3} {\circle* {2}}
\put(160,108) {\line(0,-1){26}} \put(160,78) {\line(0,-1){26}}

\put(168,110) {$4n+1$}
\put(130,144) {$1$} \put(100,144) {$2$} \put(60,144) {$n-1$} \put(40,144) {$n$}
\put(30,70) {$n+1$} \put(60,70) {$n+2$} \put(125,70) {$2n$}
\put(175,144) {$2n+1$} \put(215,144) {$2n+2$} \put(275,144) {$3n$}
\put(178,70) {$4n$} \put(212,70) {$4n-1$} \put(275,70) {$3n+1$}

\put(50,110) {$\mathbb{C}_{2n+1}$} \put(250,110) {$\mathbb{C}_{2n+1}$}

\end{picture}

The graph $\gg$ is composed of two equal cycles $\bc_{2n+1}$ having one common vertex, with
the infinite path attached to it. We put $\whe_k=e_k$, $k\ge 4n+1$,
$$ \whe_{4n-j}:=\frac{e_{j+1}+e_{2n-j}+e_{2n+j+1}+e_{4n-j}}2\,, \qquad j=0,1,\ldots,n-1. $$
Then
\begin{equation*}
\begin{split}
\ag\whe_{4n+1} &=2\whe_{4n}+\whe_{4n+2}, \quad \ag\whe_{4n}=\whe_{4n-1}+2\whe_{4n+1}, \\
\ag\whe_{4n-j} &=\whe_{4n-j-1}+\whe_{4n-j+1}, \quad j=1,2,\ldots, n-2, \\
\ag\whe_{3n+1} &=\whe_{3n+1}+\whe_{3n+2}.
\end{split}
\end{equation*}
So the Jacobi subspace and Jacobi components are now $\cl_J=\lc\{\whe_j\}_{j\ge 3n+1}$,
\begin{equation}
J(\gg):=J(\{1,0,0,\ldots\},\{a_j\}), \quad
a_j=
\left\{
  \begin{array}{ll}
    1, & j\not= n; \\
   2, & j=n
  \end{array}
\right.
\end{equation}
(cf. \eqref{jaccom5}).

The finite-dimensional component arises by putting
\begin{equation*}
\begin{split}
f_j &:=\frac{e_{n-j+1}-e_{n+j}}{\sqrt{2}}\,, \\
g_j &:=\frac{e_{3n-j+1}-e_{3n+j}}{\sqrt{2}}\,, \\
h_j &:=\frac{e_{n-j+1}+e_{n+j}-e_{3n-j+1}-e_{3n+j}}2\,, \quad j=1,\ldots,n.
\end{split}
\end{equation*}
Each of them satisfies
$$ \ag y_1=-y_1+y_2, \quad \ag y_k=y_{k-1}+y_{k+1}, \ k=2,\ldots, n-1, \ \ \ag y_n=y_{n-1}, $$
so each of linear spans is $n$-dimensional and $\ag$-invariant subspace, and (cf. \eqref{canfor5})
$$ F(\gg)=\bigoplus_{k=1}^3 J(\{-1,0,\ldots,0\},\{1\}). $$

The rest is the same as in Example \ref{kite}. The Jost polynomial is
$$ -2(z+1)u(z)=3z^{2n+1}+4z^2-1, $$
and it has two real roots such that $ -1<x_2(n)<0<x_1(n)<1$. Both of them
contribute to the discrete
spectrum of the Jacobi component
$$ \s_d(J(\gg))=\Bigl\{x_1(n)+\frac1{x_1(n)}\,, x_2(n)+\frac1{x_2(n)}\Bigr\}\,. $$
The spectrum of $F(\gg)$ is given in \eqref{fincomkite}, but now each eigenvalue has multiplicity $3$.

For spectral properties of finite propeller graphs see \cite{liuzhou}.
\end{example}

\begin{remark}
If the blades are $\bc_{2n}$, the argument goes through in exactly the same way
(see Example \ref{kite}, case 1). Moreover, the algorithm applies equally well to propellers with
$p$ {\it equal} blades for $p\ge2$ with the infinite path attached to their common vertex (``flowers'').
\end{remark}

\newpage

\begin{example}\label{umbrella} ``An umbrella''.

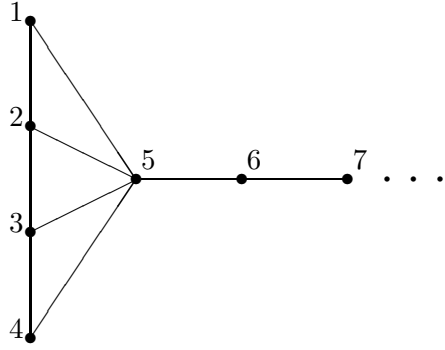
\begin{picture}(300, 160)

\multiput(100, 85) (40,0) {3} {\circle* {4}}
\multiput(102, 85) (40,0) {2} {\line(1, 0) {36}}
\multiput(195, 85) (10,0) {3} {\circle* {2}}

\multiput(60, 145) (0,-40) {4} {\circle* {4}}
\multiput(60, 143) (0,-40) {3} {\line(0,-1) {36}}

\put(60,145) {\line(2,-3) {40}}
\put(60,105) {\line(2,-1) {40}}
\put(60,65) {\line(2,1) {40}}
\put(60,25) {\line(2,3) {40}}

\put(102, 89) {$5$}
\put(142, 89) {$6$}
\put(182, 89) {$7$}

\put(52, 145) {$1$}
\put(52, 105) {$2$}
\put(52, 65) {$3$}
\put(52, 25)  {$4$}

\end{picture}

The first step is standard
$$ \whe_k:=e_k, \quad k\ge5, \qquad \whe_4:=\frac{e_1+e_2+e_3+e_4}2\,, $$
so
\begin{equation*}
\begin{split}
\ag\whe_{k} &=\whe_{k-1}+\whe_{k+1}, \quad k\ge6, \quad \ag\whe_5=2\whe_4+\whe_6, \\
\ag\whe_4 &=\frac{e_1+2e_2+2e_3+e_4}2+2\whe_5.
\end{split}
\end{equation*}
We find $\whe_3$ from $a\whe_3=\ag\whe_4-b\whe_4-2\whe_5$, where the parameters $a,b$ are
determined from the orthogonality
$\langle\whe_3,\whe_4\rangle=0$, and normalization $\|\whe_3\|=1$. It is easy to see that
$$ a=\frac12\,, \ \ b=\frac32\,, \ \ \whe_3=\frac{-e_1+e_2+e_3-e_4}2\,, \quad \ag\whe_4=\frac12\,\whe_3+\frac32\,\whe_4+2\whe_5. $$
Hence
$$ \ag\whe_3=\frac{e_1+e_4}2=-\frac12\,\whe_3+\frac12\,\whe_4, $$
so the Jacobi subspace and Jacobi components are
\begin{equation}\label{jaccom8}
\cl_J=\lc\{\whe_j\}_{j\ge 3}, \quad
J(\gg):=J\Bigl(\Bigl\{-\frac12\,,\frac32\,, 0,\ldots\Bigr\},
\Bigl\{\frac12\,, 2,1,\ldots\Bigr\}\Bigr).
\end{equation}

Next, put
$$ \whe_1:=\frac{e_1-e_4}{\sqrt2}, \quad \whe_2:=\frac{e_2-e_3}{\sqrt2}
\Longrightarrow \ag\whe_1=\whe_2, \quad \ag\whe_2=\whe_1-\whe_2, $$
and so the canonical from is
\begin{equation}\label{canfor8}
A(\gg)\simeq F(\gg)\bigoplus J(\gg),
\qquad F(\gg):=\begin{bmatrix}
0 & 1 \\
1 & -1 &
\end{bmatrix}.
\end{equation}

The Jost polynomial is now
$$ -u(z)=3z^4+3z^3+3z^2+z-1, $$
and it has two real roots $-1<x_2<0<x_1<1$. Hence, the discrete spectrum is
$$ \s_d(\gg)=\Bigl\{\frac{-1\pm\sqrt5}2\,, x_1+\frac1{x_1}\,, x_2+\frac1{x_2}\Bigr\}.$$
\end{example}

\begin{example}\label{wheel} ``A wheel''.

The wheel $W_n$ is a graph consisting of a cycle on $n$ vertices, $1,2,\ldots,n$, $n\ge3$,
and the vertex $n+1$, adjacent to each of $1,2,\ldots,n$ (cf. \cite[p. 49]{Bap}). Consider
the coupling $\gg=W_n+\bp_\infty$ with the path attached to the root $n+1$. We put
$$ \whe_j:=e_j, \ \ j\ge n+1, \qquad \whe_n:=\frac1{\sqrt{n}}\,\sum_{k=1}^n e_k, $$
and so
\begin{equation*}
\begin{split}
\ag\whe_{j} &=\whe_{j-1}+\whe_{j+1}, \quad j\ge n+2, \quad \ag\whe_{n+1}=\sqrt{n}\,\whe_n+\whe_{n+2}, \\
\ag\whe_n &=2\whe_n+\sqrt{n}\,\whe_{n+1}+2\whe_5,
\end{split}
\end{equation*}
so the Jacobi component is
\begin{equation}\label{jaccom9}
\cl_J=\lc\{\whe_j\}_{j\ge n}, \quad
J(\gg):=J(\{2,0, 0,\ldots\},\{\sqrt{n},1,1,\ldots\}).
\end{equation}

To find the finite-dimensional component, consider the adjacency operator $A(\bc_n)$ acting
in $\bc^n$.
Let
$$ \{f_1,f_2,\ldots,f_n\}, \quad f_j\in\bc^n, \quad f_n=\frac1{\sqrt{n}}\,\{1,1,\ldots,1\} $$
be the orthonormal basis of its eigenvectors. Take $\whe_j:=\{f_j,0,0,\ldots\}$.
It is a matter of a simple computation to verify that
$$ A(\gg)\,\whe_j=\{A(\bc_n)\,f_j,0,0,\ldots\}=2\cos\frac{2\pi(n-1)}{n}\,\whe_j, \quad j=1,\ldots,n-1, $$
so the canonical form is
\begin{equation}\label{canfor9}
A(\gg)\simeq F(\gg)\oplus J(\gg),
\  F(\gg):=\diag\Bigl\{2\cos\frac{2\pi}{n}\,,\ldots,2\cos\frac{2\pi(n-1)}{n}\Bigr\}.
\end{equation}

The Jost polynomial is computed in \eqref{jf1}
$$ -\sqrt{n}u(z)=(n-1)z^2+2z-1. $$
It has two roots
$$ x_2(n)=-\frac{\sqrt{n}+1}{n-1}<0<x_1(n)=\frac{\sqrt{n}-1}{n-1}\,, $$
such that $x_1(n)\in(0,1)$ for all $n\ge3$, and for $x_2(n)$ we have
$$ x_2(n)\le -1, \quad n=3,4; \qquad x_2(n)>-1, \quad n\ge5. $$
Hence, the discrete spectrum is either
$$ \s_d(W_n+\bp_\infty)=\Bigl\{2\cos\frac{2k\pi}{n}\Bigr\}_{k=1}^{n-1}\,\bigcup\,
\Bigl\{x_1(n)+\frac1{x_1(n)}\Bigr\}\,, \quad n=3,4, $$
or
$$ \s_d(W_n+\bp_\infty)=\Bigl\{2\cos\frac{2k\pi}{n}\Bigr\}_{k=1}^{n-1}\,\bigcup\,
\Bigl\{x_1(n)+\frac1{x_1(n)}\,, x_2(n)+\frac1{x_2(n)}\Bigr\}, \quad n\ge5. $$

\end{example}

\section{Graphs with several tails}

Given a finite graph $G$, we can attach several infinite paths to one or several
vertices of $G$. Although such graphs are not exactly couplings in the sense of
Definition \ref{coupl}, the algorithm suggested in Theorem \ref{algorithm} works in
this case as well (see Remark \ref{multicoupl}).

\begin{example}  ``A cycle with a double tail''.

Given the cycle $\bc_{2n+1}$ with vertices $1,2,\ldots,2n+1$, we attach two infinite paths to the
last vertex. Denote this graph by $\gg=\bc_{2n+1}+2\bp_\infty$.

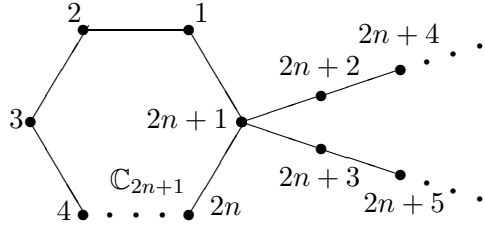
\begin{picture}(300,120)
\put(60,60) {\circle*{4}} \put(80,25) {\circle*{4}} \put(120,25) {\circle*{4}}
\put(140,60) {\circle*{4}} \put(120,95) {\circle*{4}} \put(80,95) {\circle*{4}}

\multiput(140,60)(30,10){2} {\line(3,1){28}} \multiput(140,60)(30,-10){2} {\line(3,-1){28}}
\multiput(140,60) (30,10){3} {\circle*{4}} \multiput(140,60) (30,-10){3} {\circle*{4}}
\multiput(200,80) (10,3){4} {\circle*{2}}  \multiput(200,40) (10,-3){4} {\circle*{2}}

\put(81,95) {\line(1,0){40}} \multiput(90,25)(10,0){3} {\circle*{2}} \put(90,35) {$\mathbb{C}_{2n+1}$}
\put(120,25) {\line(3,5){20}} \put(60,60) {\line(3,5){22}}
\put(120,95) {\line(3,-5){20}}  \put(60,60) {\line(3,-5){20}}
\put(104,58) {$2n+1$} \put(154,78) {$2n+2$} \put(154,36) {$2n+3$}
\put(184,90) {$2n+4$} \put(186,26) {$2n+5$}
\put(122,98) {$1$} \put(74,98) {$2$} \put(52,58) {$3$} \put(70,23) {$4$} \put(128,25) {$2n$}
\end{picture}

We proceed as in Example \ref{kite}, case 2, (see Remark \ref{multicoupl}) and put
$$ \wte_{n+j}:=\frac{e_{n-j+1}+e_{n+j}}{\sqrt{2}}\,, \quad
\wte_j:=\frac{e_{n-j+1}-e_{n+j}}{\sqrt{2}}\,, \quad j=1,2,\ldots,n, $$
and
$$ \wte_{2n+k}:=\frac{e_{2n+k}+e_{2n+k+1}}{\sqrt{2}}\,, \quad k=1,2,\ldots. $$
Clearly, the adjacency operator on the $\ag$-invariant subspace $\widetilde\cl=\lc\{\wte_j\}_{j\ge 1}$
acts exactly as the one in Example \ref{kite}, so the spectrum of $\ag\vert\widetilde\cl$ is known.

The rest is obvious. The orthogonal complement $\ell^2\ominus\widetilde\cl$ is spanned by the system
of vectors
$$ \widetilde f_{k}:=\frac{e_{2n+k}-e_{2n+k+1}}{\sqrt{2}}\,, \quad k=1,2,\ldots, $$
and $\ag\vert\ell^2\ominus\widetilde\cl=J_0$. So
\begin{equation}
A(\gg)\simeq F(\gg)\bigoplus J(\gg)\bigoplus J_0,
\end{equation}
$F(\gg)$, $J(\gg)$ are given in \eqref{canfor5}.
\end{example}

\begin{example}\label{cycletwotails} ``A cycle with two tails''.

\begin{picture}(300,120)
\put(60,40) {\circle*{4}} \put(60,80) {\circle*{4}} \put(95,100) {\circle*{4}}
\put(130,80) {\circle*{4}} \put(130,40) {\circle*{4}} \put(95,20) {\circle*{4}}

\multiput(170,80)(40,0){2} {\circle*{4}}  \multiput(225,80)(10,0){3} {\circle*{2}}
\multiput(170,40)(40,0){2} {\circle*{4}}  \multiput(225,40)(10,0){3} {\circle*{2}}
\multiput(132,80)(40,0){2} {\line(1,0){36}} \multiput(132,40)(40,0){2} {\line(1,0){36}}
\put(130,78) {\line(0,-1){36}} \put(95,20) {\line(5,3){35}} \put(60,80) {\line(5,3){35}}
\put(95,20) {\line(-5,3){35}} \put(95,100) {\line(5,-3){35}} \put(85,50) {$\mathbb{C}_{2n+1}$}
\multiput(60,70)(0,-10){3} {\circle*{2}}
\put(98,102) {$1$} \put(50,78) {$2$} \put(78,10) {$2n-1$} \put(26,36) {$2n-2$}
\put(124,88) {$2n+1$}  \put(164,88) {$2n+3$} \put(204,88) {$2n+5$}
\put(126,26) {$2n$}    \put(164,26) {$2n+2$} \put(204,26) {$2n+4$}
\end{picture}

Given the cycle $\bc_{2n+1}$ with vertices $1,2,\ldots,2n+1$, we attach two infinite paths to the
last two vertices. Denote this graph by $\gg=\bc_{2n+1}+\bp_\infty+\bp_\infty$.
We put $\wte_1=e_n$,
\begin{equation*}
\begin{split}
\wte_j &=\frac{e_{n+j-1}+e_{n-j+1}}{\sqrt2}\,, \quad j=2,3,\ldots,n, \\
\wte_{j} &=\frac{e_{2j-2}+e_{2j-1}}{\sqrt2}\,, \quad j\ge n+1.
\end{split}
\end{equation*}
The subspace $\widetilde{\cl}:=\lc\{\wte_j\}_{j\ge1}$ is $A(\gg)$-invariant and
\begin{equation*}
\begin{split}
\ag\wte_1=\sqrt2\,\wte_2, \ \ \ag\wte_2=\sqrt2\,\wte_1+\wte_3, \ \ \ag\wte_k=\wte_{k-1}+\wte_{k+1},
\ \ 3\le k\le n, \\
\ag\wte_{n+1}=\wte_n+\wte_{n+1}+\wte_{n+2}, \quad \ag\wte_k=\wte_{k-1}+\wte_{k+1}, \quad  k\ge n+2.
\end{split}
\end{equation*}
So $A(\gg)\vert\,\widetilde{\cl}=J(\{b_j\},\{a_j\})$ with
\begin{equation}\label{jaccomp2}
b_j=
\left\{
  \begin{array}{ll}
    1, & j=n+1, \\
    0, & j\not=n+1.
  \end{array}
\right.
\qquad
a_j=
\left\{
  \begin{array}{ll}
    \sqrt2, & j=1, \\
    1, & j\not=1.
  \end{array}
\right.
\end{equation}

Next, put
$$ h_j:=\frac{e_{n+j}-e_{n-j}}{\sqrt2}\,, \ \ 1\le j\le n-1,\quad
h_j:=\frac{e_{2j}-e_{2j+1}}{\sqrt2}\,,
\ \ j\ge n.
$$
It is clear that $\ch=\lc\{h_j\}_{j\ge1}$ is $A(\gg)$-invariant, $A(\gg)\vert\,\ch=J_0$, and
$\widetilde{\cl}\oplus\ch=\ell^2$. Finally we have the following canonical form
\begin{equation}\label{canfor11}
A(\gg)\simeq J(\{b_j\},\{a_j\})\bigoplus J_0,
\end{equation}
$J(\{b_j\},\{a_j\})$ is in \eqref{jaccomp2}.

The Jost polynomial is computed in \eqref{jf4}
$$ -\sqrt2u(z)=z^{2n+1}+z^2+z-1, $$
and, by the Descarte's rule, we see that $u$ has one real root $0<x_1(n)<1$. Hence,
the discrete spectrum  is
$$ \s_d(\gg)=x_1(n)+\frac1{x_1(n)}\,, $$
and the absolutely continuous spectrum has multiplicity $2$.
\end{example}

\medskip

Surprisingly enough, the case when $p\ge1$ infinite rays are attached to {\it each} vertex of a
finite graph $G$ with vertices $\{1,2,\ldots,n\}$ (so there are $pn$ such rays altogether) is easy
to handle, and the spectrum of such graph can be found explicitly in terms of the spectrum of $G$.
Denote such graph by $\gg=G+\bp_\infty(p)$. We label the vertices along the rays (off $G$) as
$$ R_i=\{m_j+i\}_{j=1}^\infty, \quad m_j:=n+(j-1)pn, \quad i=1,2,\ldots,pn. $$

\begin{theorem}\label{sungraph}
Given a finite graph $G$ with vertices $1,2,\ldots,n$ and the spectrum
$\s(G)=\{\l_j\}_{j=1}^n$, let $\gg=G+\bp_\infty(p)$, $p\in\bn$.
Then the adjacency operator $A(\gg)$ is unitarily equivalent to the orthogonal sum
\begin{equation}\label{canfor12}
\begin{split}
A(\gg) &\simeq \bigoplus_{j=1}^n J(\l_j,\sqrt{p})
\bigoplus\Bigl(\bigoplus_{i=1}^{(p-1)n}J_0\Bigr), \\
J(\l_j,\sqrt{p})&:=J(\{\l_j,0,0,\ldots\},\{\sqrt{p},1,1,\ldots\}).
\end{split}
\end{equation}
The spectrum of $\gg$ is
\begin{equation}\label{specgen}
\s(\gg)=[-2,2]\bigcup\bigcup_{j=1}^n \s_d(J(\l_j,\sqrt{p})).
\end{equation}
\end{theorem}
\begin{proof}
Put
$$ g_j(k):=\frac1{\sqrt{p}}\,\sum_{l=1}^p e_{m_j+(k-1)p+l}, \quad k=1,2,\ldots,n, \quad j\in\bn, $$
so $\{g_j(k)\}_{j,k}$ is the orthonormal sequence in $\ell^2$. It is easy to see from the way of labeling
of the vertices that the adjacency operator $\ag$ acts on these vectors as
\begin{equation}\label{actong}
\begin{split}
\ag g_j(k) &=g_{j-1}(k)+g_{j+1}(k), \quad j=2,3,\ldots  \\
\ag g_1(k) &=\sqrt{p}\,e_k+g_2(k).
\end{split}
\end{equation}
Next, by the definition
\begin{equation}
\ag\,e_k=\{A(G)e_k^{(n)},0,0,\ldots\}+\sqrt{p}\,g_1(k), \quad k=1,2,\ldots,n,
\end{equation}
where $\{e_k^{(n)}\}_{k=1}^n$ is the standard basis in $\bc^n$.

To construct the canonical basis we invoke the eigenvectors of $A(G)$
$$ f_k^{(n)}:=\sum_{q=1}^n \z_{qk}\,e_q^{(n)}\,, \qquad \z=\|\z_{qk}\|_{q,k=1}^n $$
is a unitary matrix, so $A(G)f_k^{(n)}=\l_k f_k^{(n)}$, $k=1,\ldots,n$. With such a notation at hand we put
\begin{equation*}
\begin{split}
\wte_0(k) &:=\sum_{q=1}^n \z_{qk}\,e_q=(f_k^{(n)},0,0\ldots), \\
\wte_j(k) &:=\sum_{q=1}^n \z_{qk}\,g_j(q)\,, \quad j\in\bn, \quad k=1,\ldots,n.
\end{split}
\end{equation*}
It is clear from \eqref{actong} that $\ag\,\wte_{j}(k)=\wte_{j-1}(k)+\wte_{j+1}(k)$, $j\ge2$,
$$ \ag\wte_1(k)=\sum_{q=1}^n \z_{qk}\,(\sqrt{p}\,e_q+g_2(q))=\sqrt{p}\,\wte_0(k)+\wte_2(k), $$
and
\begin{equation*}
\begin{split}
\ag\wte_0(k) &=\sum_{q=1}^n \z_{qk}\,\ag\,e_q=\Bigl(\sum_{q=1}^n \z_{qk}\,A(G)e_q^{(n)},0,0,\ldots\Bigr)
+\sqrt{p}\,\wte_1(k) \\
&=\l_k\wte_0(k)+\sqrt{p}\,\wte_1(k).
\end{split}
\end{equation*}
Hence the subspace $\cl:=\lc\{\wte_j(k)\}_{1\le k\le n, j=0,1\ldots}$ is $\ag$-invariant and
$$ \ag\vert\cl\simeq \bigoplus_{k=1}^n J(\l_k,\sqrt{p}). $$

It is not hard to see that on the orthogonal complement
$$ \ag\vert\ell^2\ominus\cl\simeq\bigoplus_{k=1}^{(p-1)n} J_0. $$
The proof is complete.
\end{proof}

In the case $p=1$ this result is proved in \cite{Niz14}.

To find the discrete spectrum of $J(\l_k,\sqrt{p})$ one has to solve the Jost equation \eqref{jf1}
\begin{equation}\label{quadr2}
(p-1)x^2+\l_k x-1=0,
\end{equation}
pick its spectral roots, that is, the roots in $(-1,1)$, and then take their Zhukovsky images.
The simplest case is $p=1$, when each $\l_k$ with $|\l_k|>1$ generates one eigenvalue $\l_k+\l_k^{-1}$
of $\ag$. For $p\ge3$ an elementary analysis of equation \eqref{quadr2} shows that there are two spectral
roots for $|\l_k|<\frac{p(p-2)}{p-1}$, and there is one such root otherwise.

\begin{example}
Consider the case $p=2$. Equation \eqref{quadr2} has the roots
$$ x_{\pm}=\frac{\l_k\pm\sqrt{\l_k^2+4}}2\,, \quad x_-<0<x_+. $$
It is easy to see that there is exactly one spectral root for $\l_k\not=0$, and no such roots for $\l_k=0$.
The discrete spectrum now looks as follows
$$ \s_d(\gg)=\bigcup_{\l_k\not=0} {\rm sign}\l_k\,\sqrt{\l_k^2+4}. $$

\end{example}

\medskip

We complete the section with two examples wherein two-sided Jacobi matrices arise naturally.

\begin{example} ``A double infinite star graph''.

Denote by $S_{k,\infty}$ the infinite star-like graph with $k\ge2$ infinite rays emanating from the common
root (cf. Example \ref{starmult}). The main object under consideration is the coupling of
two such graphs with the bridge connecting their roots, so $\gg=S_{p,\infty}+S_{q,\infty}$.

We label the vertices of the ``right'' star $S_{p,\infty}$ along the rays by positive integers as
$$ \{1,2,p+2,\ldots\}, \ \{1,3,p+3,\ldots\}, \ldots, \ \{1,p+1,2p+1,\ldots\}, $$
so the root is $1$. Similarly, we number the vertices of the ``left'' graph $S_{q,\infty}$ by negative
integers, so the root is $-1$ and the vertices along the rays are
$$ \{-1,-2,-q-2,\ldots\}, \ \{-1,-3,-q-3,\ldots\}, \ldots, \ \{-1,-q-1,-2q-1,\ldots\}. $$
The underlying $\ell^2$ space is $\ell^2(\bz_0)$, $\bz_0:=\bz\backslash\{0\}$, and the standard basis is
$\{e_j\}_{j\in\bz_0}$.

We construct an orthonormal system $\{\wte_j\}_{j\in\bz_0}$ by
\begin{equation*}
\begin{split}
\wte_1 &:=e_1, \qquad \ \ \wte_k:=\frac1{\sqrt{p}}\,\sum_{i=2}^{p+1} e_{(k-2)p+i}, \ \quad k\ge2; \\
\wte_{-1} &:=e_{-1}, \qquad \wte_j:=\frac1{\sqrt{q}}\,\sum_{i=2}^{q+1} e_{(j+2)q-i}, \quad j\le -2.
\end{split}
\end{equation*}
It is easy to see that $\ag$ acts as
\begin{equation*}
\begin{split}
\ag\wte_1 &=\wte_{-1}+\sqrt{p}\,\wte_2, \quad \qquad \qquad \ \ag\wte_{-1} =\sqrt{q}\,\wte_{-2}+\wte_1, \\
\ag\wte_2 &=\sqrt{p}\,\wte_{1}+\wte_3,  \qquad \qquad \qquad \ag\wte_{-2} =\wte_{-3}+\sqrt{q}\,\wte_{-1}, \\
\ag\wte_k &=\wte_{k-1}+\wte_{k+1}, \ \ k\ge3, \qquad \ag\wte_{-j} =\wte_{-j-1}+\wte_{-j+1}, \ \ j\ge3.
\end{split}
\end{equation*}
Hence the subspace $\widetilde\cl=\lc\{\wte_j\}_{j\in\bz_0}$ is $\ag$-invariant and
\begin{equation}
\ag\vert\widetilde\cl=J(p,q)=J(\{0\},\{a_j\}_{j\in\bz}), \quad
a_j=\left\{
      \begin{array}{ll}
        \sqrt{p}, & j=1; \\
        \sqrt{q}, & j=-1; \\
        1, & j\not=\pm1.
      \end{array}
    \right.
\end{equation}

To supplement the system $\{\wte_j\}_{j\in\bz_0}$ to the orthonormal basis we proceed in a standard way
\begin{equation*}
\begin{split}
\wte_k^{(l)} &:=\sum_{i=2}^{p+1} \xi_{i-1,l}\,e_{(k-1)p+i}, \quad l=1,\ldots,p, \quad k=1,2,\ldots, \qquad \ \
\|\xi_{rs}\|=\cf_p; \\
\wte_j^{(l)} &:=\sum_{i=2}^{q+1} \eta_{i-1,l}\,e_{(j+1)q-i}, \quad l=1,\ldots,q, \ \ j=-1,-2,\ldots, \ \
\|\eta_{rs}\|=\cf_q.
\end{split}
\end{equation*}
Finally,
\begin{equation}
\ag\simeq J(p,q)\bigoplus\Bigl(\bigoplus_{i=1}^{p+q-2} J_0\Bigr).
\end{equation}

To find the discrete spectrum we apply Example \ref{twojm2} with $a_{-1}=\sqrt{q}$, $a_1=\sqrt{p}$, so
characteristic equation \eqref{quadr} looks
$$ Q(y)=(p-1)(q-1)y^2-(p+q-1)y+1=0. $$
It is easy to see that this equation always has tow positive roots $0<y_-<y_+$.
In the cases $\min(p,q)=2$ or $p=q=3$ we have $y_+\ge1$, so only the first root has the spectral meaning.
Otherwise, $0<y_-<y_+<1$, so both of them are the spectral roots. Hence,
$$ \s_d(J(p,q))=\pm\Bigl(\sqrt{y_-}+\frac1{\sqrt{y_-}}\Bigr), \qquad \min(p,q)=2 \ \ {\rm or} \ \ p=q=3, $$
or
$$ \s_d(J(p,q))=\left\{\pm\Bigl(\sqrt{y_-}+\frac1{\sqrt{y_-}}\Bigr),
\pm\Bigl(\sqrt{y_+}+\frac1{\sqrt{y_+}}\Bigr)\right\} $$
for the rest of the values $p,q\ge2$. Here
$$ y_{\pm}=\frac{p+q-1\pm\sqrt{(p-q)^2+2(p+q)-3}}{2(p-1)(q-1)}\,. $$

Note that if the bridge has weight $d$, the above argument ends up with the Jacobi matrix
$$ J(p,q,d)=J(\{0\},\{a_j\}_{j\in\bz}), \quad
a_j=\left\{
      \begin{array}{ll}
        \sqrt{p}, & j=1; \\
        \sqrt{q}, & j=-1; \\
        d, & j=0,
      \end{array}
    \right. \quad a_j=1, \ \ j\not=0,\pm1,
$$
for which the spectrum can be found explicitly (see Introduction).

The example was studied in \cite{Niz14} by using a different model based on one-sided Jacobi matrices
\eqref{defjac}, with no explicit expressions provided.
\end{example}

\begin{example} ``Infinite regular trees''.

The Bethe--Cayley trees $\bb_d$ of degree $d\in\bn$ (or Bethe lattices in physical literature) are among the
most notable infinite trees. For the introduction to the spectral theory on such trees see
\cite[Chapter 10]{SiSz}. We consider here the coupling of two copies of $\bb_d$ with the bridge connecting
their roots, $\gg=\bb_d+\bb_d$. Note that $\gg$ is an infinite regular tree of degree $d+1$. We number the
vertices of $\gg$ as in the previous example (the vertices of the ``right'' copy by positive integers,
and of the ``left'' copy by negative integers). So, the vertices of $k$'s generation in the right copy
(the root forms the 1st generation) are labeled with
$$ \{l_{k-1}+1,l_{k-1}+2,\ldots,l_k\}, \qquad l_k:=\sum_{j=1}^{k-1} d^j=\frac{d^{k}-1}{d-1}\,,
\quad k\in\bn. $$
Again, the underlying Hilbert space is $\ell^2(\bz_0)$ with the basis $\{e_j\}_{j\in\bz_0}$.

Put $\wte_{\pm1}:=e_{\pm1}$,
$$ \wte_k:=\frac1{\sqrt{d^{k-1}}}\,\sum_{i=l_{k-1}+1}^{l_k} e_i, \qquad
\wte_{-k}:=\frac1{\sqrt{d^{k-1}}}\,\sum_{j=l_{k-1}+1}^{l_k} e_{-j}, \quad  k=2,3,\ldots. $$
The operator $\ag$ acts as
\begin{equation*}
\begin{split}
\ag\,\wte_1 &=\whe_{-1}+\sqrt{d}\,\wte_2, \qquad \ag\,\wte_{-1} =\sqrt{d}\,\whe_{-2}+\wte_1, \\
\ag\,\wte_k &=\sqrt{d}\,\wte_{k-1}+\sqrt{d}\,\wte_{k+1}, \qquad |k|\ge2.
\end{split}
\end{equation*}
So the subspace $\widetilde\cl=\lc\{\wte_j\}_{j\in\bz_0}$ is $\ag$-invariant and
\begin{equation}\label{4.9}
\ag\vert\widetilde\cl=J(\gg)=\sqrt{d}\,J(\{0\},\{a_j\}_{j\in\bz}), \quad
a_j=\left\{
      \begin{array}{ll}
        1/\sqrt{d}, & j=0; \\
        1, & j\not=0
      \end{array}
    \right.
\end{equation}

The structure of operator $\ag$ on the orthogonal complement is the same as in the case of $\bb_d$ (see
\cite{Br07}, \cite[Theorem 10.2.2]{SiSz}), so finally
\begin{equation}
\ag\simeq J(\gg)\bigoplus\Bigl(\bigoplus_{i=1}^{\omega_d} \sqrt{d}\,J_0\Bigr), \qquad
\omega_d=
\left\{
  \begin{array}{ll}
    \infty, & d\ge2; \\
    0, & d=1.
  \end{array}
\right.
\end{equation}

To find the discrete spectrum of $\ag$ note that $J(\gg)$ \eqref{4.9} is the subject of Example \ref{twojm1}.
It follows that the discrete spectrum of $J(\gg)$ is empty, and so
\begin{equation}\label{spbethe}
\s(\gg)=[-2\sqrt{d}, 2\sqrt{d}]
\end{equation}
and is pure absolutely continuous of infinite multiplicity.
\end{example}


\begin{remark}
Since $\bb_1=\bp_\infty$, the question arises naturally whether it is possible to find spectra of the couplings
$\gg=G+\bb_d$. The above algorithm works equally well in this situation and leads to the canonical form
$$ A(G+\bb_d)\simeq A(G+\bp_\infty)\bigoplus\Bigl(\bigoplus_{i=1}^{\omega_d} \sqrt{d}\,J_0\Bigr). $$
\end{remark}

\end{document}